\documentclass[letterpaper, 10 pt, conference]{ieeeconf}
\IEEEoverridecommandlockouts
\usepackage{graphics} % for pdf, bitmapped graphics files
\usepackage[pdftex]{graphicx}
\usepackage{epstopdf}
\usepackage[cmex10]{amsmath}
\usepackage{amssymb}  % assumes amsmath package installed
\usepackage{enumerate}
\usepackage{setspace}
\usepackage{verbatim}

\newtheorem{example}{Example}
\newtheorem{theorem}{Theorem}

\newtheorem{remark}{Remark}

\newtheorem{corollary}{Corollary}%[section]

\def\x{x}
\def\y{y}
\def\r{r}
\def\u{u}
\def\v{v}

\def\c{c}

\begin{document}

\title{Extremum Seeking Optimal Controls of Unknown Systems}
\author{ Alexander~Scheinker and David Scheinker
\thanks{{Alexander~Scheinker is at Los Alamos National Laboratory. Email: {\tt\small ascheink@lanl.gov}. David~Scheinker is at Stanford University. Email: {\tt\small dscheink@stanford.edu}.}}%}
\thanks{{This research was supported by Los Alamos National Laboratory and Stanford University.}}%
}

\maketitle

%%%%%%%%%%%%%%%%%%%%%%%%%%%%%%%%%%%%%%%%%%%%%%%%%%%%%%%%%%%%%%%%%%%%%%%%%%%%%%%%%%%%%%%%

\begin{abstract}
We present a method for finding optimal controllers for unknown, time-varying, dynamic systems which can be re-initialized from a given initial condition repeatedly, in which the performance measure is available for sampling with noise, but analytically unknown. Such systems are present throughout industry where processes must be repeated many times, such as a voltage source which is repeatedly turned on for a fraction of a second from zero initial conditions and then turned off again, whose output must track a specific trajectory, while the system's components are slowly drifting with time due to temperature variations. For systems with convex cost functions we prove that our algorithm will produce controllers that approach the minimal cost, e.g., the cost minimizing LQR optimal controller that could have been designed analytically had the system and objective function been known. We demonstrate the algorithm's effectiveness with simulation studies of noisy and time-varying systems.\end{abstract}

%%%%%%%%%%%%%%%%%%%%%%%%%%%%%%%%%%%%%%%%%%%%%%%%%%%%%%%%%%%%%%%%%%%%%%%%%%%%%%%%%%%%%%%%

%%%%%%%%%%%%%%%%%%%%%%%%%%%%%%%%%%%%%%%%%%%%%%%%%%%%%%%%%%%%%%%%%%%%%%%%%%%%%%%%%%%%%%%%

%%%%%%%%%%%%%%%%%%%%%%%%%%%%%%%%%%%%%%%%%%%%%%%%%%%%%%%%%%%%%%%%%%%%%%%%%%%%%%%%%%%%%%%%

\section{Introduction}

\subsection{ES Background}
Extremum seeking (ES) is a model-free optimization technique which is being actively studied in the control community. Recent developments include utilizng ES for open loop control of tokamaks \cite{conf-luo09}, magnetohydrodynamic channel flow control \cite{ref-carnevale09}, a Lie bracket approximation approach for studying ES dynamics \cite{ref-durr-stan-john-11}, ES for stabilization of unknown, open-loop unstable systems \cite{ref-Sch-Krstic-TAC}, multivariable Newton-based ES for photovoltaic power optimization \cite{ref-Krstic-power}, Newton-based stochastic ES \cite{ref-Krstic-stoch}, constrained ES \cite{ref-Guay2}, electromagnetic actuator control \cite{ref-ben1}, gain tuning for nonlinear control \cite{ref-ben2}, and a proportional-integral design approach \cite{ref-Guay1}. In this paper we study the iterative creation of optimal controllers of repeatable, analytically unknown dynamic systems with analytically unknown objective functions based on the general dither, constrained ES algorithms studied in \cite{ref-Sch-Sch,ref-Sch-Sch2}.

\subsection{Repeatable systems}
We study repeatable systems, which are always re-initialized from the same initial conditions. Examples of such systems include robotic manipulators which perform the same motion repeatedly or accelerating cavities in particle accelerators which are turned on for a fraction of a second, hundreds of times per second. For example, at the Los Alamos Linear Particle Accelerator, accelerating cavities are re-initialized from zero initial conditions at a rate of 120Hz and each run of the system lasts only $T = 0.001$ seconds, as depicted in Figure \ref{time_scale}. In the control literature, this is sometimes referred to as batch-to-batch control \cite{ref-batch1,ref-batch2}.

\begin{figure*}[!t]
	\centering
    \includegraphics[width=.8\textwidth]{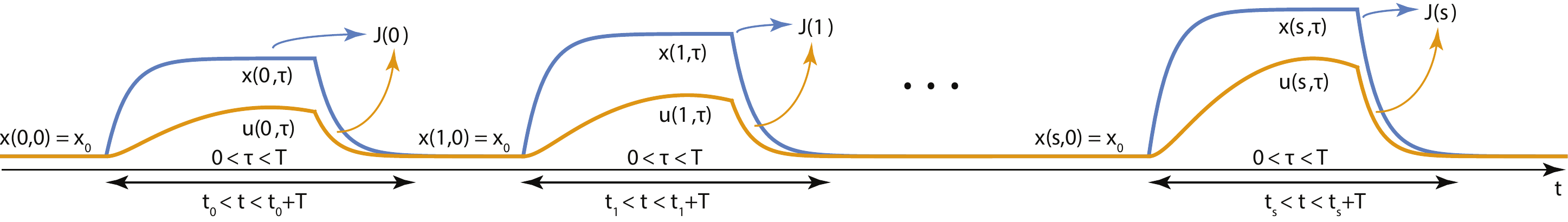}
\caption{Typical system.}
\label{time_scale}
\end{figure*}

\subsection{Optimal control of unknown systems background}

The optimal control of unknown systems has been studied with several different approaches. In \cite{ref-opt7} an iterative learning approach is utilized in which controllers are developed out of basis functions for systems without detailed knowledge. However, \cite{ref-opt7} requires knowledge of the objective function being minimized (the matrices $R$ and $Q$ as defined in (\ref{LQR_track_cost}) are utilized as part of the algorithm), does not provide an analytic proof of convergence, and relies on training data whereas the algorithm presented here can operate directly on samples of the unknown objective function values. The work in \cite{ref-Krstic-batch} is closely related to our work presented here. In \cite{ref-Krstic-batch} a novel discrete-time stochastic averaging and stochastic ES approach are used to iteratively optimize open loop control sequences for discrete-time linear systems with scalar inputs. In \cite{ref-Krstic-ILC} an ES-based iterative learning control (ILF) approach is developed which is applicable to a wide range of optimization and tracking problems and unlike typical ILC results does not require system knowledge. In \cite{ref-opt1}, online system identification is followed by training of neural networks off line to solve infinite horizon control problems. In \cite{ref-opt3}, an adaptive control scheme for control of unknown systems with infinite horizon cost function is presented. In \cite{ref-opt2} discrete time systems with infinite horizon cost functions are studied via adaptive dynamic programming, developing output feedback optimal controllers. In \cite{ref-opt4}, unknown, nonaffine nonlinear discrete-time systems are studied for which optimal controllers are developed via adaptive dynamic programming. In \cite{ref-opt5}, control affine nonlinear discrete-time systems are studied and feedback controllers are designed using neural networks. In \cite{ref-opt6}, optimal open-loop controllers are derived for unknown discrete-time linear systems via ES. 

\subsection{Contributions of this paper}
As in some of the above results, our work is applicable to noisy, uncertain systems without knowledge of the system dynamics or of analytic knowledge of the objective function, working purely based on noisy scalar measurements of an unknown objective function. The contributions of this work are: \\
1). Unlike iterative learning control approaches, we do not require knowledge of a reference trajectory to track, we do not require system dynamics knowledge, we can handle noisy and time-varying systems. \\
2). Our work is applicable to time-varying systems. In the case of unknown time varying vector-valued systems of the form $\dot{x}=f(x,t)+g(x,t)u$, we iteratively construct controllers which find the global minimum of convex objective functions. \\
3). We do not require knowledge of the analytic form of the objective function, only noisy measurement are used. \\
4). In the case of unknown time varying vector-valued systems of the form $\dot{x}=A(t)x+B(t)u$, we iteratively construct controllers for convex objective functions and prove that they track, arbitrarily closely, the time-varying LQR-type optimal controllers which could be designed if $A(t)$, $B(t)$ and the objective function were known. \\
5). In the case of unknown, vector-valued, time invariant systems such as $\dot{x}=Ax+Bu$, when full state measurements, $x(t)$, are available, we are able to iteratively calculate the unique optimal feedback gain, $K(t)$, such that $u = -K(t)x$ is the optimal feedback controller relative to an analytically unknown convex objective function, such as the finite-horizon, continuous-time LQR. By creating the optimal feedback controller, our results are initial-condition independent, unlike the feedforward results of previous work. Furthermore, this same approach will provide the optimal linear quadratic tracker for a chosen trajectory.

\subsection{Summary of the results of the paper}
We consider an analytically unknown optimal control problem with trajectory $x$ and cost $J$. These depend on controller $u(s)$ where $s=0,1,\dots,N$ is an index for a sequence of controllers approaching an optimal controller.
\begin{eqnarray}
	\frac{dx(s,\tau)}{d\tau} &=& f(x(s,\tau),u(s,\tau)), \quad x(s,0)=x_0, \label{dynamics_iter} \\
	J(u(s,\tau)) &=&  F(x(s,T)) + \int_0^T G(x(s,\tau),u(s,\tau))d\tau, \label{cost_iter} \\
	u(s,\tau) &=& \sum_{j=1}^m a_j(s) \phi_j(\tau), \label{control_iter}
\end{eqnarray}
where the functions $\phi_j(\tau)$ are a subset of any basis for $L^p[0,T]$, such as, for example, the well known Fourier basis
\begin{equation}
	\phi_j(\tau) = \cos(\nu_j\tau) \quad \mathrm{or} \quad \sin(\nu_j \tau), \quad \nu_j = 2\pi j/T.
\end{equation}
Our iterative procedure for finding an optimal controller for system (\ref{dynamics_iter})-(\ref{control_iter}) is initialized at $s=0$ with a controller $u(0,\tau) = 0$, i.e., $a_n = 0$ for all $n$. With this controller, the system evolves for $\tau \in [0,T]$ and results in $J(0)$. The parameters of the controller are then updated as
\begin{equation}
		a_j(s+1) = a_j(s) + \Delta \sqrt{\alpha \omega_j}\cos \left ( \omega_j s \Delta + k J(s) \right ). \label{a_update}
\end{equation}
For $s>0$, the system is re-initialized to $x_0$, evolves using $u(s,\tau)$ for $\tau \in [0,T]$, and results in $J(s+1)$. In this paper, we show that this iterative method results in a minimization of $J$, which has a unique global minimum when convex.

In the case of optimal feedback control design of the form $u=-K(t)x$, we create the $n \times n$ matrix $K(t)$ where each function $k_{l,m}(t)$ is constructed out of linear combinations of basis functions of the form (\ref{control_iter}). For an n-dimensional system, we converge to the unique optimal $K(t)$ when the ES algorithm is used to simultaneously optimize $n$ objective functions $J = \sum_{j=1}^{n}J_j$ of the form (\ref{cost_iter}) for n-linearly independent initial conditions of a fixed linear system. Once the algorithm has converged we have the universal optimal feedback controller for any initial condition.

%
%%
%%%
%%%%
%%%%%
%%%%%%
%%%%%%%
%%%%%%%%%%%%%%%%%%%%%%%%%%%%%%%%%%%%%%%%%%%%%%%%%%%%%%%%%%%%%%%%%%%%%%%%%%%%%%%%%%%%%%%%
%%%%%%%%%%%%%%%%%%%%%%%%%%%%%%%%%%%%%%%%%%%%%%%%%%%%%%%%%%%%%%%%%%%%%%%%%%%%%%%%%%%%%%%%

\section{Problem Setup and ES Background}

%%%%%%%%%%%%%%%%%%%%%%%%%%%%%%%%%%%%%%%%%%%%%%%%%%%%%%%%%%%%%%%%%%%%%%%%%%%%%%%%%%%%%%%%
%%%%%%%%%%%%%%%%%%%%%%%%%%%%%%%%%%%%%%%%%%%%%%%%%%%%%%%%%%%%%%%%%%%%%%%%%%%%%%%%%%%%%%%%
%%%%%%%
%%%%%%
%%%%%
%%%%
%%%
%%
%

Consider system system (\ref{dynamics_iter})-(\ref{control_iter}). The iterative procedure for finding the optimal controller is: \\
1). Choose $k>0$, $\alpha>0$, and a set of distinct frequencies, $\left \{ \omega_n \right \}_{n=0}^m$, of the form $\omega_n = \omega_0 r_n$, where $\omega_0 \gg 1$ and $r_i \neq r_j$ for $i \neq j$. Choose a time step, $\Delta = \frac{2\pi}{N \omega_0}$, with $N \gg 1$. \\
2). At the initial step $s=0$, set the values of the coefficients $a_n(0)$ to zero or to the best guess for what they should be. This defines the controller $u(a(0),\tau)$ as given by (\ref{control_iter}). Apply this initial controller $u(0,\tau)$ to the system (\ref{dynamics_iter}) over the time interval $[0, T]$ and record the performance index, $J(0)$, as defined by (\ref{cost_iter}). \\
3). At step $s>0$, update the values of the coefficients according to (\ref{a_update}). Note that instead of $\cos()$, one may choose $\sin()$, or square waves, or triangle waves, or any oscillatory functions satisfying some weak limits as described in Theorem 1 of \cite{ref-Sch-Sch}. \\
4). The new coefficients, $a(s+1)$, define an updated controller, $u(s+1,\tau)$, which is then applied to the system (\ref{dynamics_iter}), for the time interval $[0, T]$ and record the performance index, $J(s+1)$, as defined by (\ref{cost_iter}).

To make clear the choice for the dynamics of the algorithm given by (\ref{a_update}), we state a result that is a direct Corollary of Theorem 1 in \cite{ref-Sch-Sch}.

\begin{corollary}\label{cor1}
Consider the vector valued system
\begin{equation}
	\frac{da_j(t)}{dt} = \sqrt{\alpha\omega_j}\cos \left ( \omega_j t + k \hat{J}(a,t) \right ), \label{adot0}
\end{equation}
where $\omega_j = \omega_0 r_j$ and $r_i \neq r_j$ for $i\neq j$, $a=(a_1,\dots,a_n) \in \mathbb{R}^n$, $J : \mathbb{R}^n \times \mathbb{R} \rightarrow \mathbb{R}$ is twice continuously differentiable with respect to $a$, and $\hat{J}=J+n(t)$ is a noise-corrupted measurement of the unknown function $J$. For large $\omega_0$, the average system whose trajectory is closely related to (\ref{adot0}) is:
\begin{equation}
	\frac{d\bar{a}_j}{dt} = -\frac{k\alpha}{2}\frac{\partial J(\bar{a},t)}{\partial \bar{a}_j}, \label{abdot2}
\end{equation}
a gradient descent of the actual unknown function $J(a,t)$ with respect to $a$, despite noise-corrupted measurements.
\end{corollary}

The reason for choosing the parameter update equation (\ref{a_update}) can now be seen as a finite difference approximation of the dynamics (\ref{adot0}) which will perform a gradient descent of the unknown objective function $J$ by finding the optimal controller. For $\Delta \ll 1$, we can write the finite difference approximation of the derivative (\ref{adot0}) as
\begin{equation}
	\frac{a_j(t + \Delta) - a_j(t)}{\Delta} = \sqrt{\alpha \omega_n}\cos \left ( \omega_n t + k J(a(t),t) \right ). \label{update1_v1}
\end{equation}
Taking samples of (\ref{update1_v1}) at time steps $t_s = s\Delta$, we rewrite the right side of (\ref{update1_v1}) as
\begin{equation}
	\sqrt{\alpha \omega_j}\cos \left ( \omega_j s\Delta + k J(a(s\Delta),s\Delta) \right ). \label{update1_v2}
\end{equation}
Now we have iteratively defined sequences of values of $a_j$ and $J$, based only on $a_j(0)$ and $J(0)$, $\left \{ a_j(s) \right \}_{s=0}^\infty$ and $  \left \{ J(s) \right \}_{s=0}^\infty$, where, for notational convenience we refer to $a_j(s\Delta)$ as $a_j(s)$ and to $J(a(s\Delta),s\Delta)$ as $J(s)$. We then have
\begin{eqnarray}
	\frac{a_n(s+1) - a_n(s)}{\Delta} = \sqrt{\alpha \omega_n}\cos \left ( \omega_n s \Delta + k J(s) \right ), \label{update1_v3}
\end{eqnarray}
which gives the iterative update equations (\ref{a_update}). While our analytic results hold for dynamics of the form (\ref{adot0}), In hardware, sufficiently large $\omega_0$ and small $\Delta$, result in convergence.

With the approach described above, from the iteratively updated parameters' point of view, in the limit as $\Delta \rightarrow 0$, the overall system and update dynamics take on a continuous, two time scale $(t,\tau)$ form:
\begin{eqnarray}
	\frac{dx(t,\tau)}{d\tau} &=& f(x(t,\tau),u(t,\tau)), \quad x(t,0)=x_0, \label{dynamics_iter2} \\
    u(t,\tau) & = & \sum_{j=1}^m a_j(t) \phi_j(\tau), \label{controller_iter2} \\
    \frac{da_j(t)}{dt} & = & \sqrt{\alpha\omega_j}\cos \left ( \omega_j t + k J(u(t,\tau)) \right )  \label{a_iter2} \\
    J(u(t,\tau)) &=&  F(x(t,T)) + \int_0^T G(x(t,\tau),u(t,\tau))d\tau. \label{cost_iter2}
\end{eqnarray}
Before we go on to prove our technical results about systems of the form (\ref{dynamics_iter2})-(\ref{cost_iter2}), we clarify the use of two time scales, $t \in \mathbb{R}^+$ and $\tau \in [0,T]$, we explicitly work $J(t)$ in a simple example of the above. 
\begin{example}
Consider the system
\begin{eqnarray}
	\frac{dx(\tau,t)}{d\tau} &=& u(\tau,t), \quad x(0,t) = x_0, \quad \tau \in [0,T], \\
	u(\tau,t) &=& a(t)\psi(\tau),\quad \psi(\tau) \in L^p[0,T].
\end{eqnarray}
We can expanding the objective function 
\begin{equation}
	J=\int_{0}^{T}x^2(\tau,t)d\tau + \int_{0}^{T} u^2(\tau,t)d\tau
\end{equation}
in terms of the solution for $x(\tau,t)$ as
\begin{equation}
	J = \int_{0}^{T}\left [ \left ( x_0 + a(t)\int_{0}^{\tau}\psi(s)ds \right )^2 + a^2(t)\psi^2(\tau) \right ]d\tau.
\end{equation}
If we evolve $a(t)$ according to the dynamics
\begin{equation}
	\frac{da}{dt} = \sqrt{\alpha\omega}\cos \left ( \omega t + k \hat{J}(a) \right ), \quad \hat{J} = J + n(t),
\end{equation}
for large $\omega$, the average dynamics are
\begin{equation}
	\frac{d\bar{a}}{dt} = -\frac{k\alpha}{2} \frac{dJ(\bar{a})}{d\bar{a}},
\end{equation}
a gradient descent of $J$, where in this simple case
\begin{equation}
	\frac{1}{2}\frac{dJ}{da} = \int_{0}^{T} \left ( a\left [ \int_{0}^{\tau}\psi(s)ds \right ]^2 + a\psi^2(\tau) + \psi(\tau)  \right )d\tau. \label{gradient}
\end{equation}
\end{example}

%
%%
%%%
%%%%
%%%%%
%%%%%%
%%%%%%%
%%%%%%%%%%%%%%%%%%%%%%%%%%%%%%%%%%%%%%%%%%%%%%%%%%%%%%%%%%%%%%%%%%%%%%%%%%%%%%%%%%%%%%%%
%%%%%%%%%%%%%%%%%%%%%%%%%%%%%%%%%%%%%%%%%%%%%%%%%%%%%%%%%%%%%%%%%%%%%%%%%%%%%%%%%%%%%%%%

\section{Proof of Convergence to the Optimal Controller}

%%%%%%%%%%%%%%%%%%%%%%%%%%%%%%%%%%%%%%%%%%%%%%%%%%%%%%%%%%%%%%%%%%%%%%%%%%%%%%%%%%%%%%%%
%%%%%%%%%%%%%%%%%%%%%%%%%%%%%%%%%%%%%%%%%%%%%%%%%%%%%%%%%%%%%%%%%%%%%%%%%%%%%%%%%%%%%%%%
%%%%%%%
%%%%%%
%%%%%
%%%%
%%%
%%
%

\begin{theorem}\label{Optimality}

Consider the system (\ref{dynamics_iter2})-(\ref{cost_iter2}) and suppose that $J(u,t)$ is differentiable and convex in $u$ and that there exists a $u^*(t)$ that minimizes $J(u,t)$ over $L^p([0,T])$ at each time $t$. Furthermore, assume that there exists a bound $M>0$ on the rate of change of $u^\star(t)$ such that $\left \| du^\star(t)/dt \right \| < M$. Consider a set of distinct frequencies, $\left \{ \omega_n \right \}_{n=0}^m$, of the form $\omega_n = \omega_0 r_n$, where $\omega_0 \gg 1$ and $r_i \neq r_j$ for $i \neq j$. For each $\omega$, let $a_{\omega}(t) = (a_{\omega,1}(t),...,a_{\omega,m}(t))$ denote the vector of solutions to the following system 
\begin{equation}
	\frac{da_{\omega,j}}{d t} = \sqrt{\alpha \omega_j}\cos \left ( \omega_j \tau + k J(u(a_\omega(t)),t) \right ), \quad  a(0) = 0.
\end{equation}
For any $\epsilon>0$, there exists $k>0$, $\alpha>0$, $\omega^\star$, $m^\star$, and $t^*,T^*>0$ such that for each $\omega>\omega^\star$, $m>m^\star$, and $t\in [t^*,t^*+T^*]$,
\begin{equation}
 | J(a_\omega(t),t) - J(u^*(t),t) | < \epsilon.
\end{equation}
\end{theorem}

\begin{remark}
The key technical component of the proof is a consequence of Theorem 1 in \cite{ref-Sch-Sch}. To simplify the statement and proof of the theorem, we assume the existence of an optimal solution. General sufficient conditions for this are available in the classical works of Rockafeller and related papers \cite{RockafellerI} - \cite{RockafellerII}.
\end{remark}

\begin{proof}
Fix $\epsilon > 0$ and a basis $\left \{ \phi_j(\tau) \right \}_1^\infty$ for $L^p[(0,T)]$. Note that each controller $u(t,\tau)$ is uniquely defined by its coefficients $a(t) = \{a_n(t)\}_1^\infty$ in the basis $\phi = \{\phi_n(\tau)\}_1^\infty$, i.e. 
\[u(t,\tau) = a(t)\cdot\phi(\tau) = \sum_{j=1}^\infty a_j(t)\phi(\tau).\]
The cost $J(u,t)$ is thus a function of the vector of coefficients $a(t)$ and may be written as $J(u(a(t)),t)$.

Consider the system of $L^p$ valued differential equations defined by the gradient descent of the coefficients of $J$
\begin{eqnarray}
	\frac{da}{dt} &=& -\frac{k\alpha}{2}\frac{\partial J(u,t)}{\partial a} = 
    -\frac{k\alpha}{2}\frac{\partial J(a \cdot \phi,t)}{\partial a}.  \label{GDUm}
\end{eqnarray}
$J(u,t)$ is convex in $u$ by assumption and $u$ is linear in $a$ by definition. Since the composition of a convex and a linear function is convex, $J(u(a),t)$ is convex in $a$. A $u^*(t)$ that minimizes $J(u,t)$ exists for each $t$ by assumption and there exists a bound, $M>0$ such that $\left \| du^\star(t)/dt \right \| < M$ by assumption. Thus, for any $\epsilon>0$, there exists $k\alpha>0$ such that system (\ref{GDUm}) converges to within $\epsilon$ of such a $u^*(t)$, i.e. $\displaystyle{\lim_{t\to\infty} \left \| a(t) - a^*(t) \right \| < \epsilon}$. Let $a^*(t)$ denote the coefficients of $u^*(t)$ in the basis $\phi_n(\tau)$, i.e. 
\[u^*(t) = \sum_{j=1}^\infty a^*_j(t)\phi_j(\tau).\]
Since $J$ is continuous in $t$ and $a$, and $a$ is continuous in $t$, $J(a(t),t)$ is continuous in $t$. Thus, there exists $t^*$ such that for each $t>t^*$ we have $|J(a(t),t) - J(a^*(t),t)| < \frac{\epsilon}{2}$.

We will complete the proof by showing that: 1). For $t>t^*$ and a sufficiently large $m$ the restriction of $a(t)$ to an m-dimensional subspace of $L^p[(0,T)]$ is close to $u^*(t)$. 2). For sufficiently large $\omega$, the original system $a_\omega(t)$ is sufficiently close to $a(t)$.

Each choice of $m$ corresponds to a subspace of $L^p([0,T])$, denoted $U_m$, spanned by $\{\phi_n(\tau)\}_1^m$. For $m$ fixed, denote the projection of $u^*(t)$ to $U_m$ with $u_m^*(t) = \sum_{j=1}^m a_j^*(t)\phi_j(\tau)$ and 
the projection of $a(t)$ to $U_m$ with $a_m(t) = \{a_j(t)\}_1^m$. Since $J(a(t),t)$ is continuous in $t$, there exists an $m^*$ such that for $t>t^*$ and any $m > m^*$ we have $|J(a_m(t),t)-J(u^*(t),t)|<\frac{\epsilon}{2}$. 

Theorem 1 in \cite{ref-Sch-Sch} which implies that for each $\delta >0$ and $t_0,T\geq0$, there exists $\omega^\star$ such that for each $\omega>\omega^\star$, the trajectories 
$a_\omega(t)$ and $a_m(t)$ satisfy $\max_{t\in \left [ t_0, t_0 + T \right ]} \left \| a_\omega(t) - \alpha_m(t) \right \| < \delta$.

Since $J$ is continuous in $\alpha_m$, there exists a $\delta$ 
such that $\left \| a_\omega(t) - \alpha_m(t) \right \| < \delta$ implies that 
$|J(a_\omega(t)) - J(\alpha_m(t))| < \frac{\epsilon}{2}$. Thus, for $T^*>0$, $m > m^*$, $t\in [t^*,t^*+T^*]$, and $\omega>\omega^\star$, by the triangle inequality, $|J(a_\omega(t)) - J(a^*(t))|$ is bounded by \[|J(a(t),t) - J(a_m(t),t)| +  |J(a_m^*(t),t) - J(a^*(t),t)| < \epsilon. \]
\end{proof}

%
%%
%%%
%%%%
%%%%%
%%%%%%
%%%%%%%
%%%%%%%%%%%%%%%%%%%%%%%%%%%%%%%%%%%%%%%%%%%%%%%%%%%%%%%%%%%%%%%%%%%%%%%%%%%%%%%%%%%%%%%%
%%%%%%%%%%%%%%%%%%%%%%%%%%%%%%%%%%%%%%%%%%%%%%%%%%%%%%%%%%%%%%%%%%%%%%%%%%%%%%%%%%%%%%%%

\section{Application to Linear Quadratic Tracker}\label{ESLQRT}

%%%%%%%%%%%%%%%%%%%%%%%%%%%%%%%%%%%%%%%%%%%%%%%%%%%%%%%%%%%%%%%%%%%%%%%%%%%%%%%%%%%%%%%%
%%%%%%%%%%%%%%%%%%%%%%%%%%%%%%%%%%%%%%%%%%%%%%%%%%%%%%%%%%%%%%%%%%%%%%%%%%%%%%%%%%%%%%%%
%%%%%%%
%%%%%%
%%%%%
%%%%
%%%
%%
%

\begin{figure*}[!t]
    \centering
    \includegraphics[width=.232\textwidth]{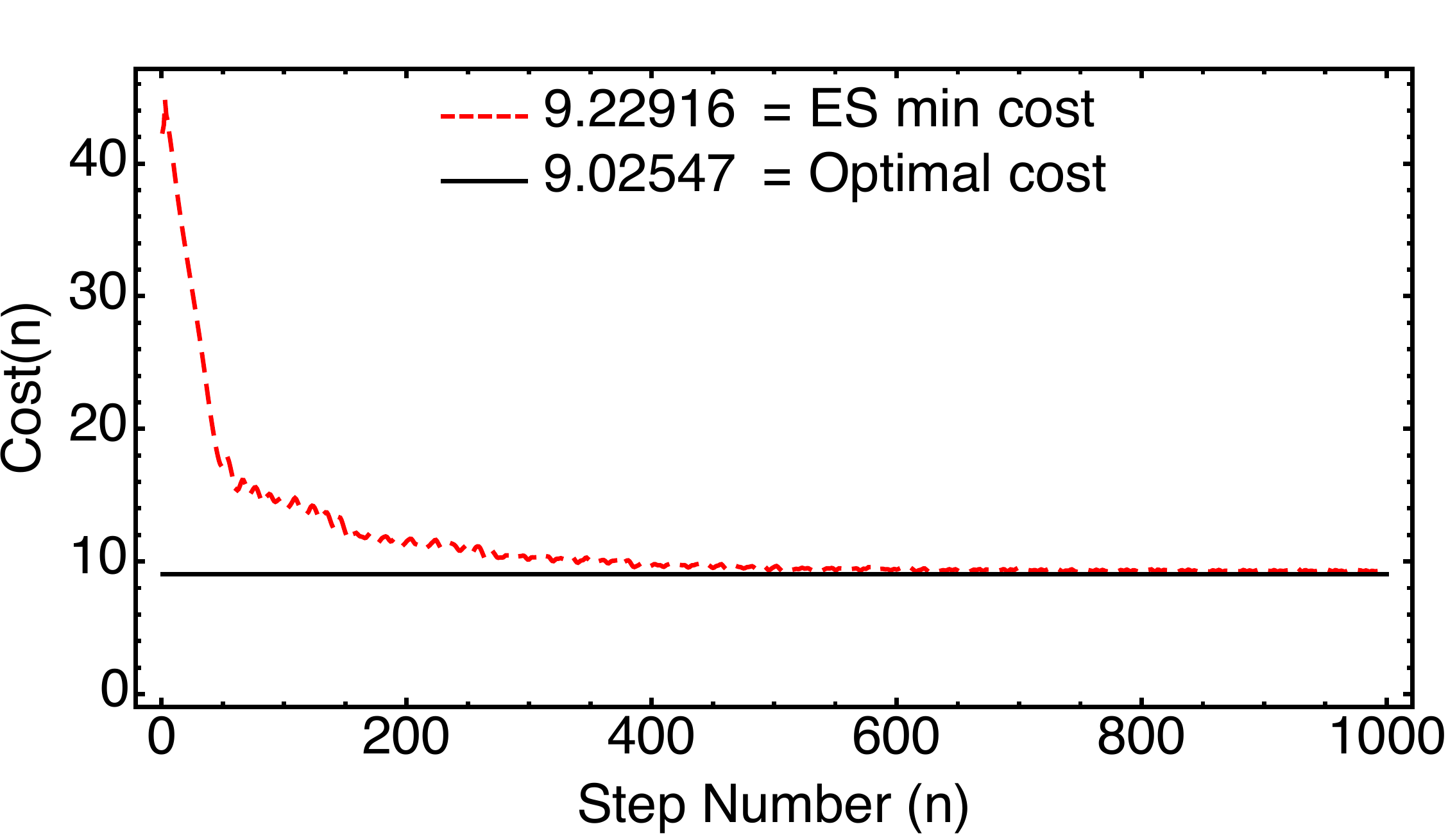} \
    \includegraphics[width=.232\textwidth]{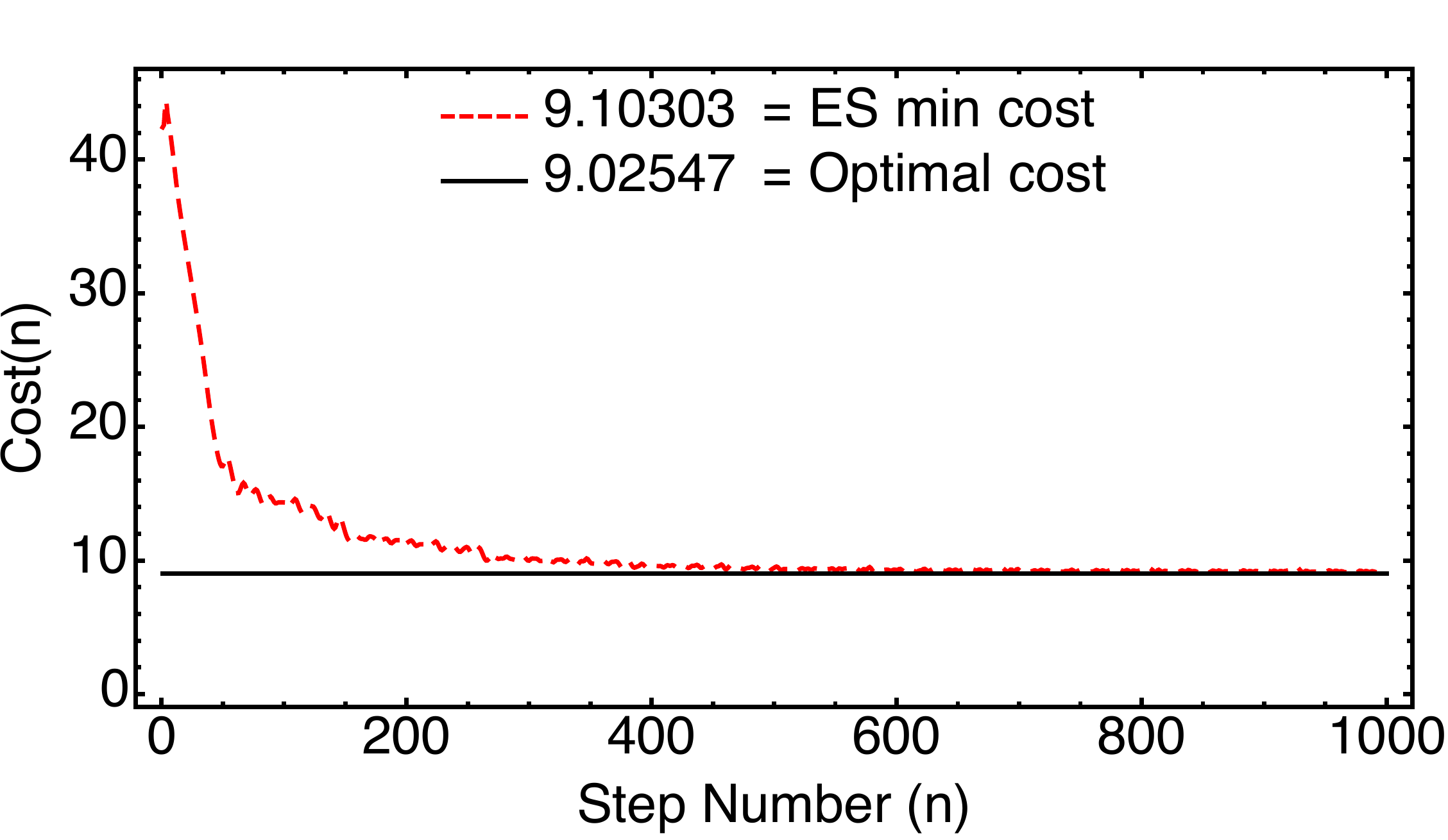} \
    \includegraphics[width=.22\textwidth]{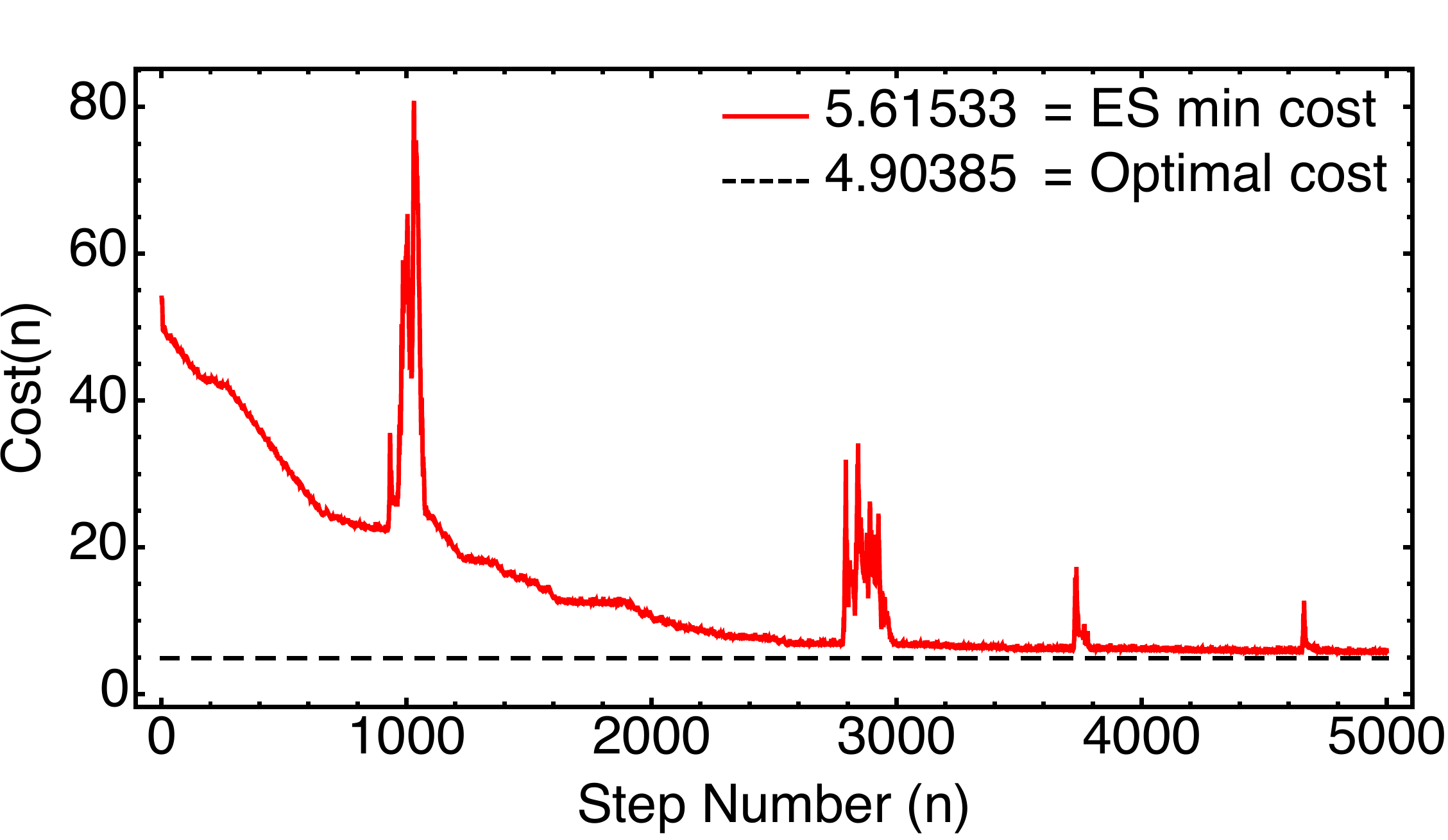} \
    \includegraphics[width=.232\textwidth]{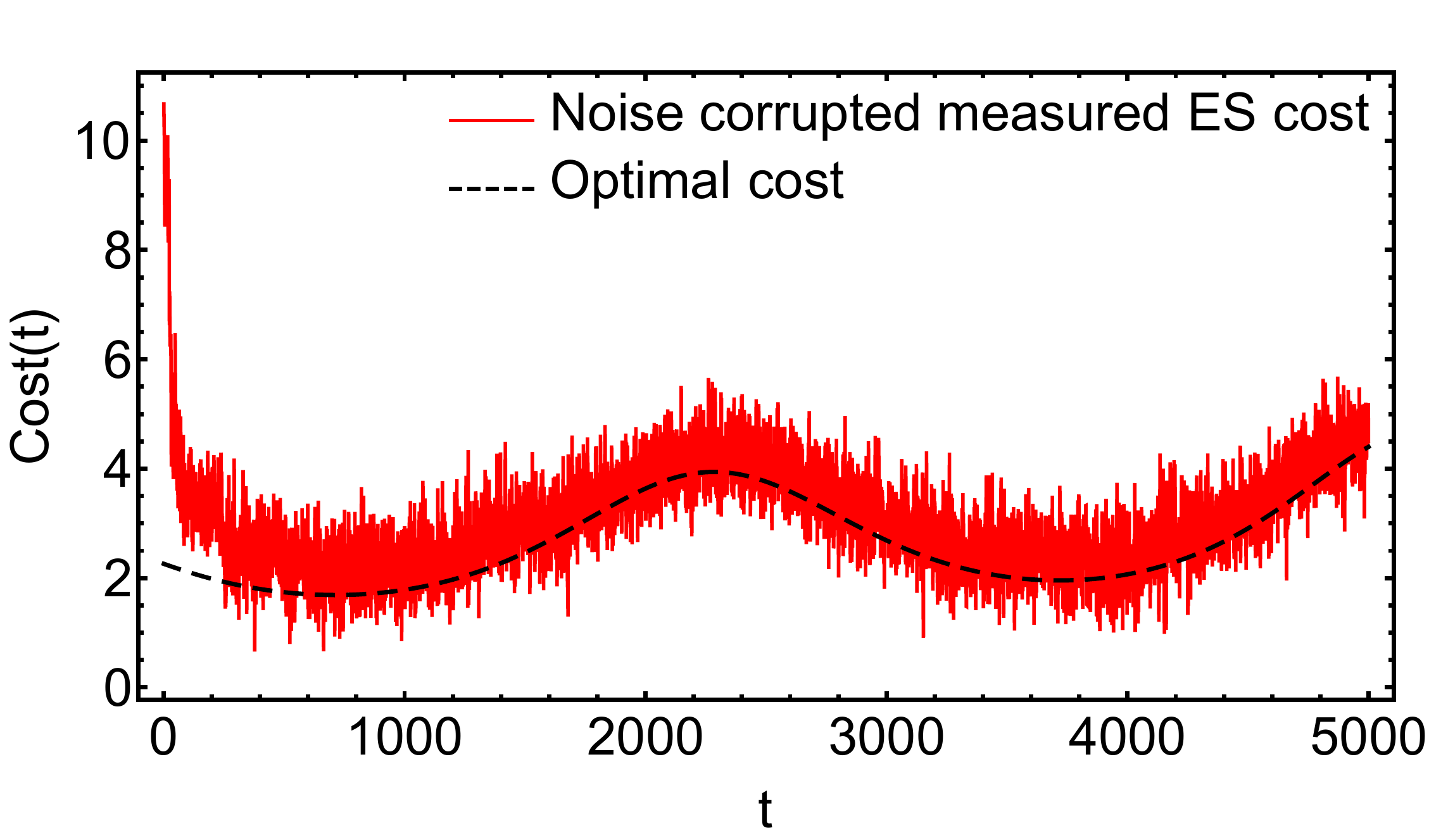} \\
    \includegraphics[width=.232\textwidth]{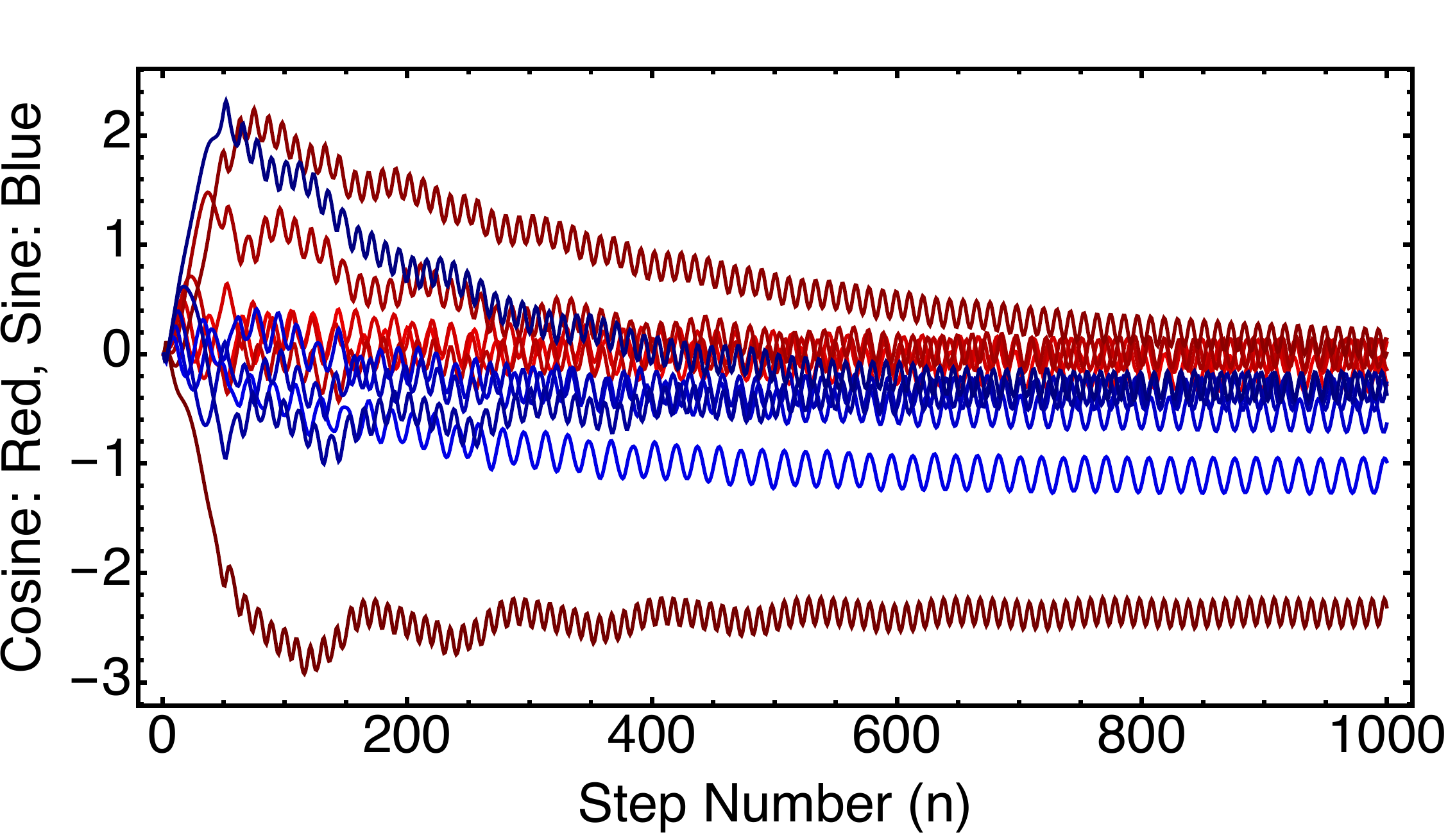} \
    \includegraphics[width=.232\textwidth]{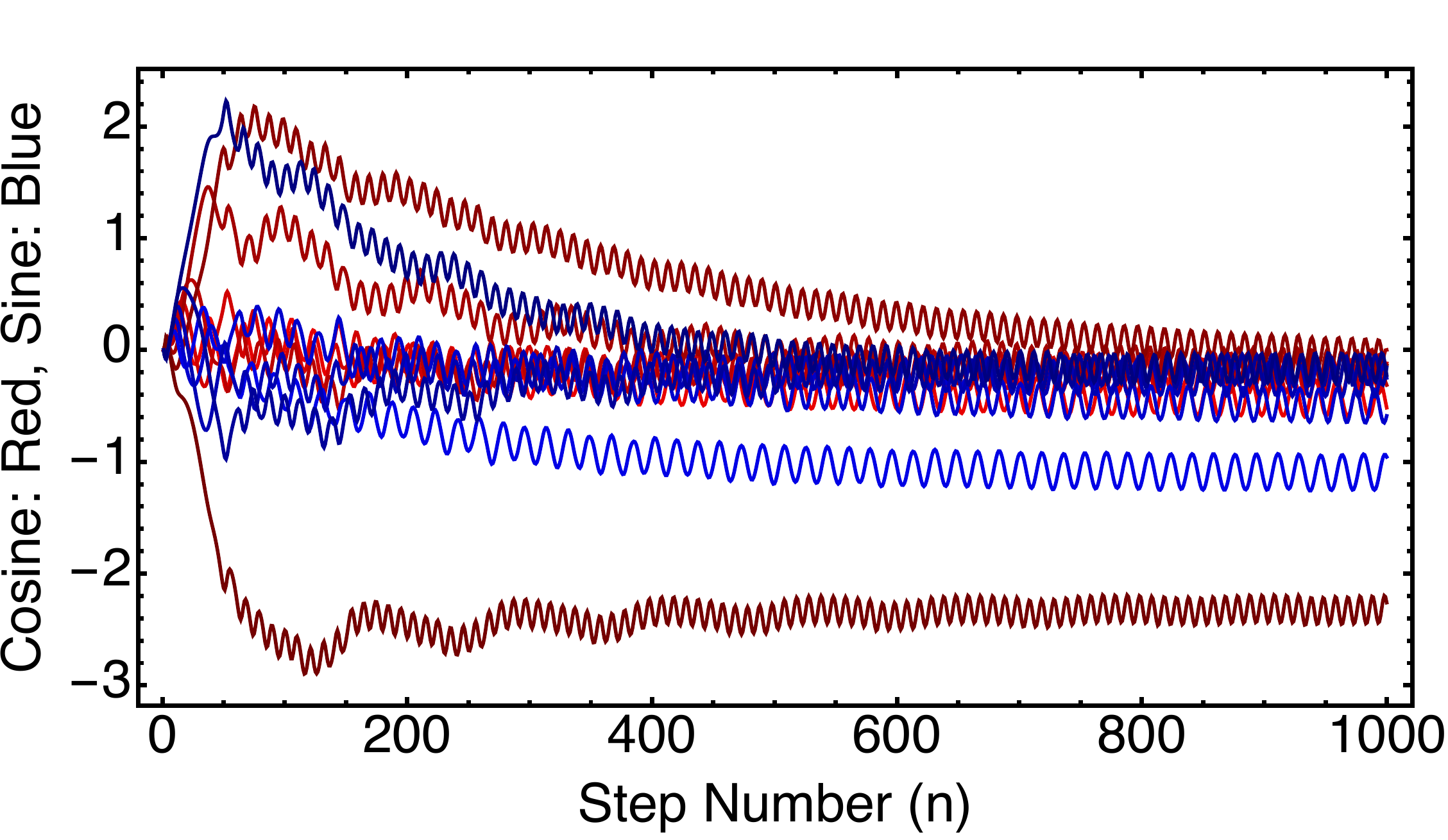} \
    \includegraphics[width=.232\textwidth]{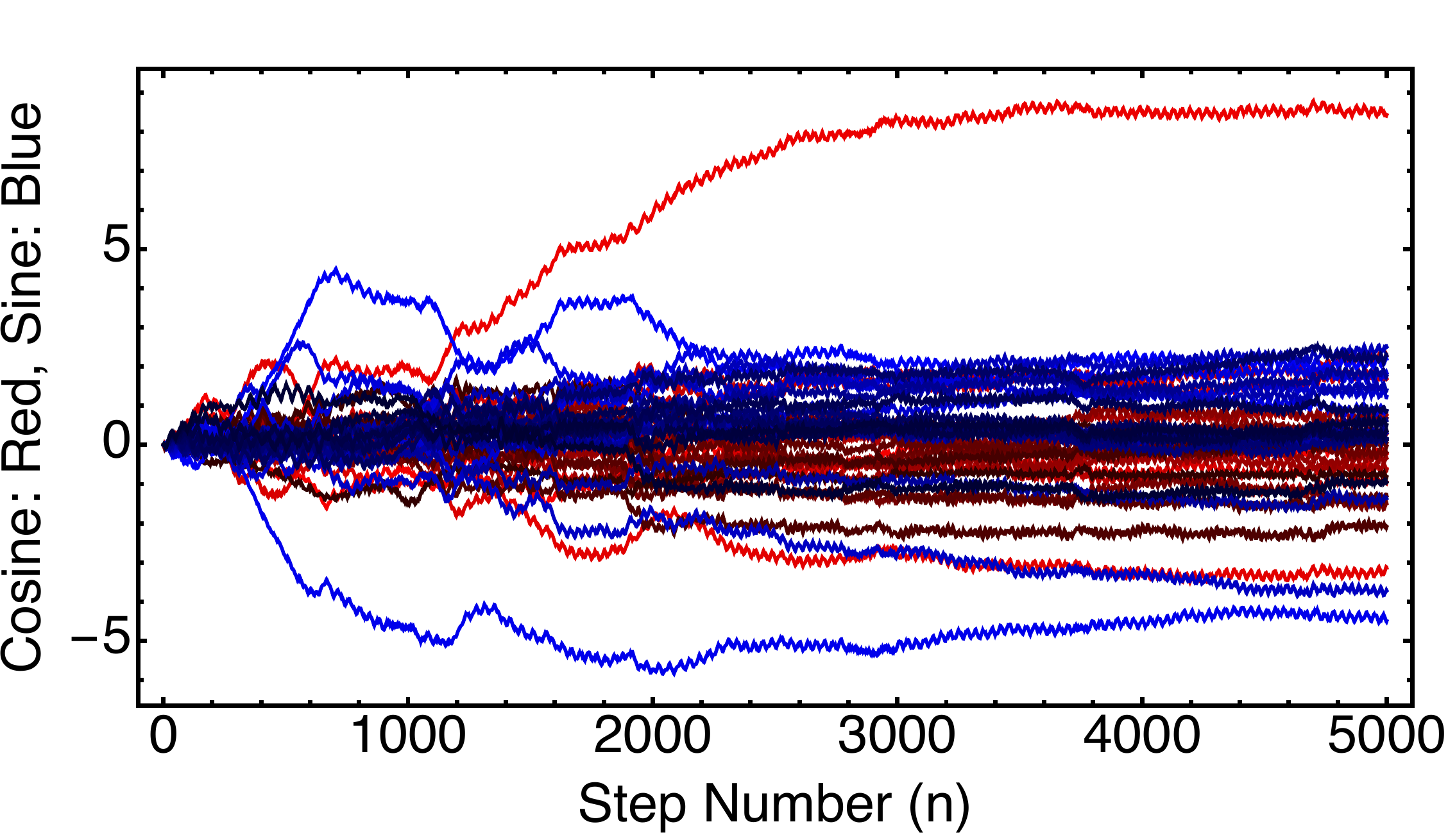} \
    \includegraphics[width=.232\textwidth]{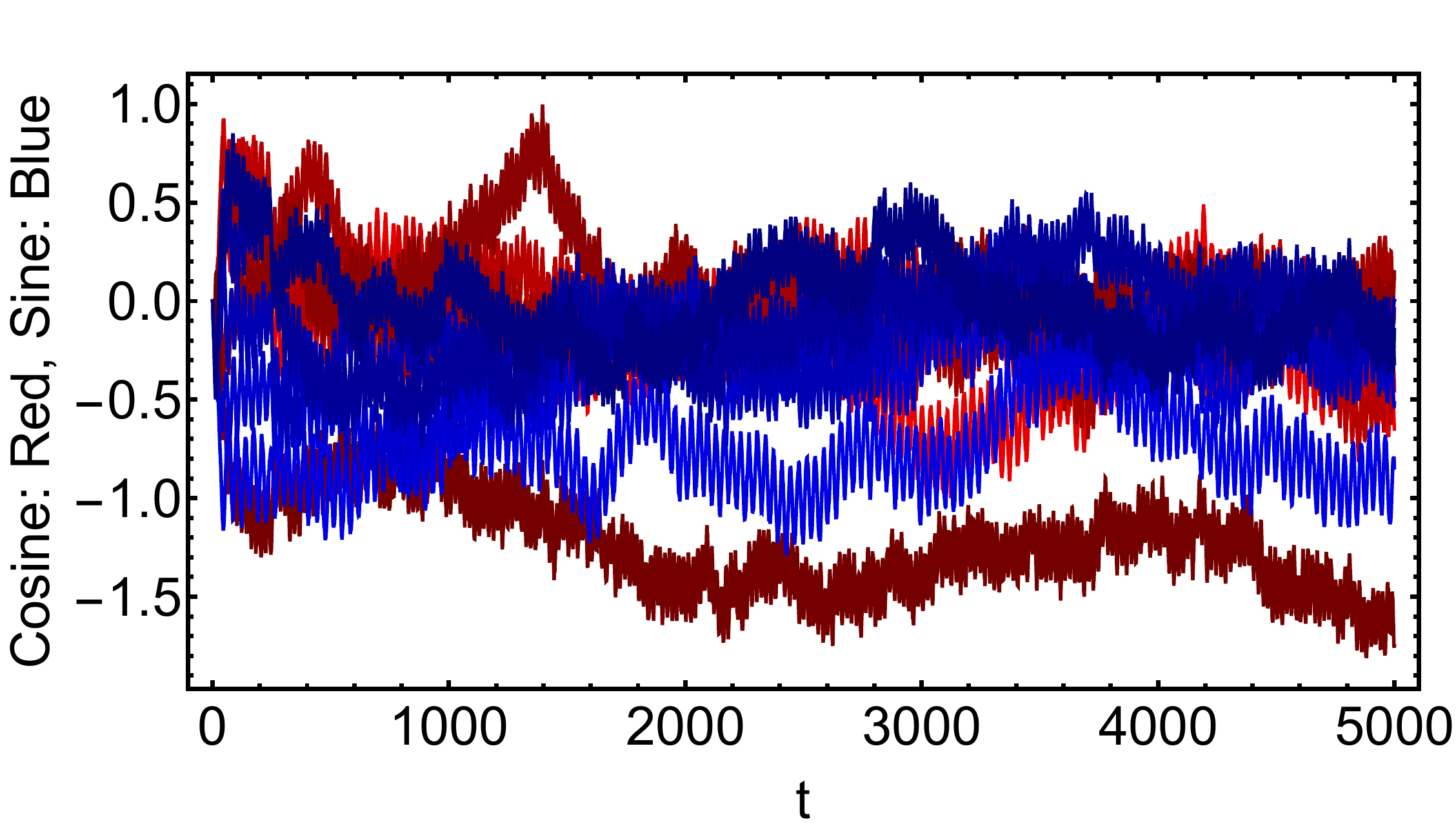} \\
    \includegraphics[width=.232\textwidth]{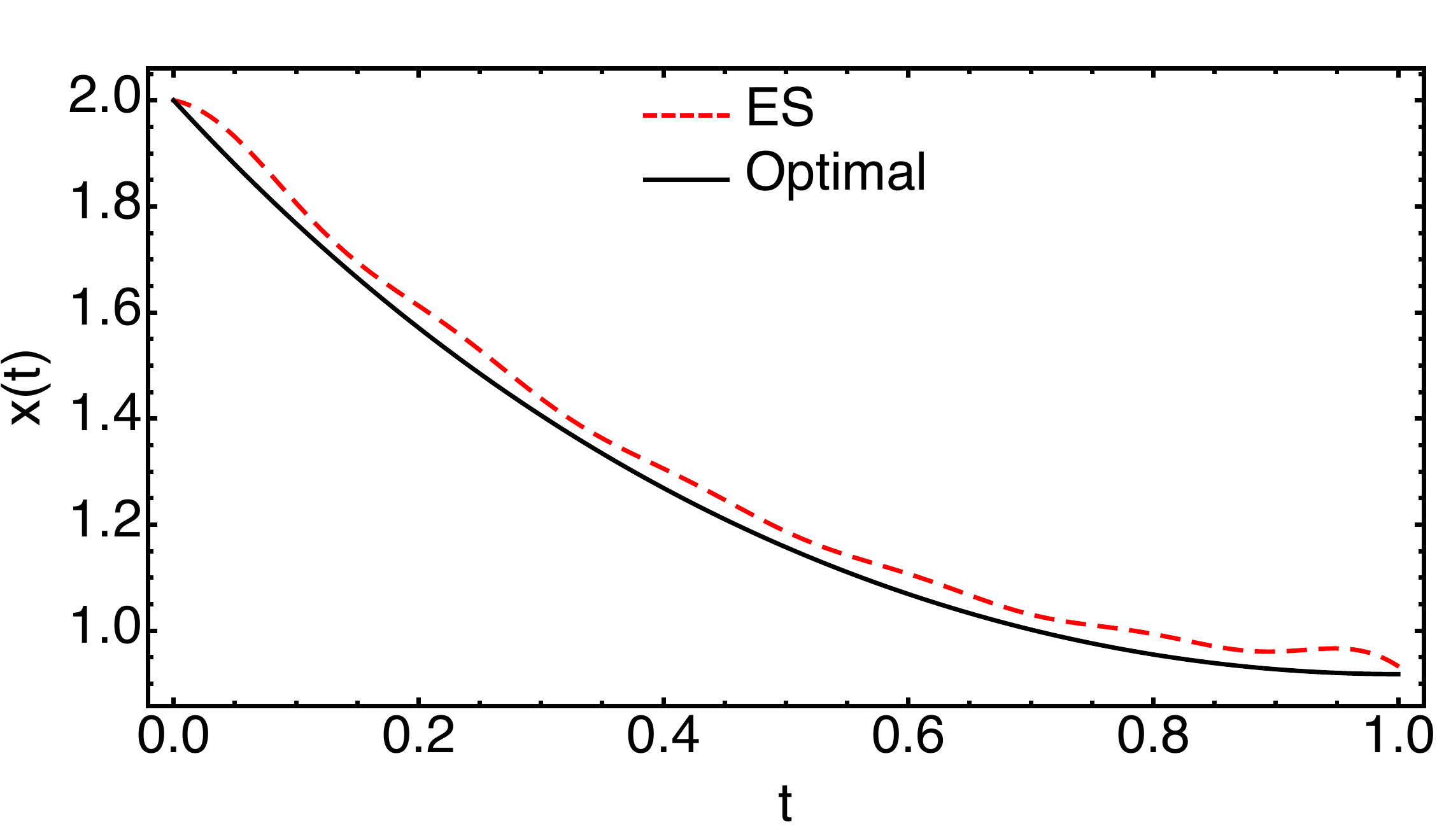} \
    \includegraphics[width=.232\textwidth]{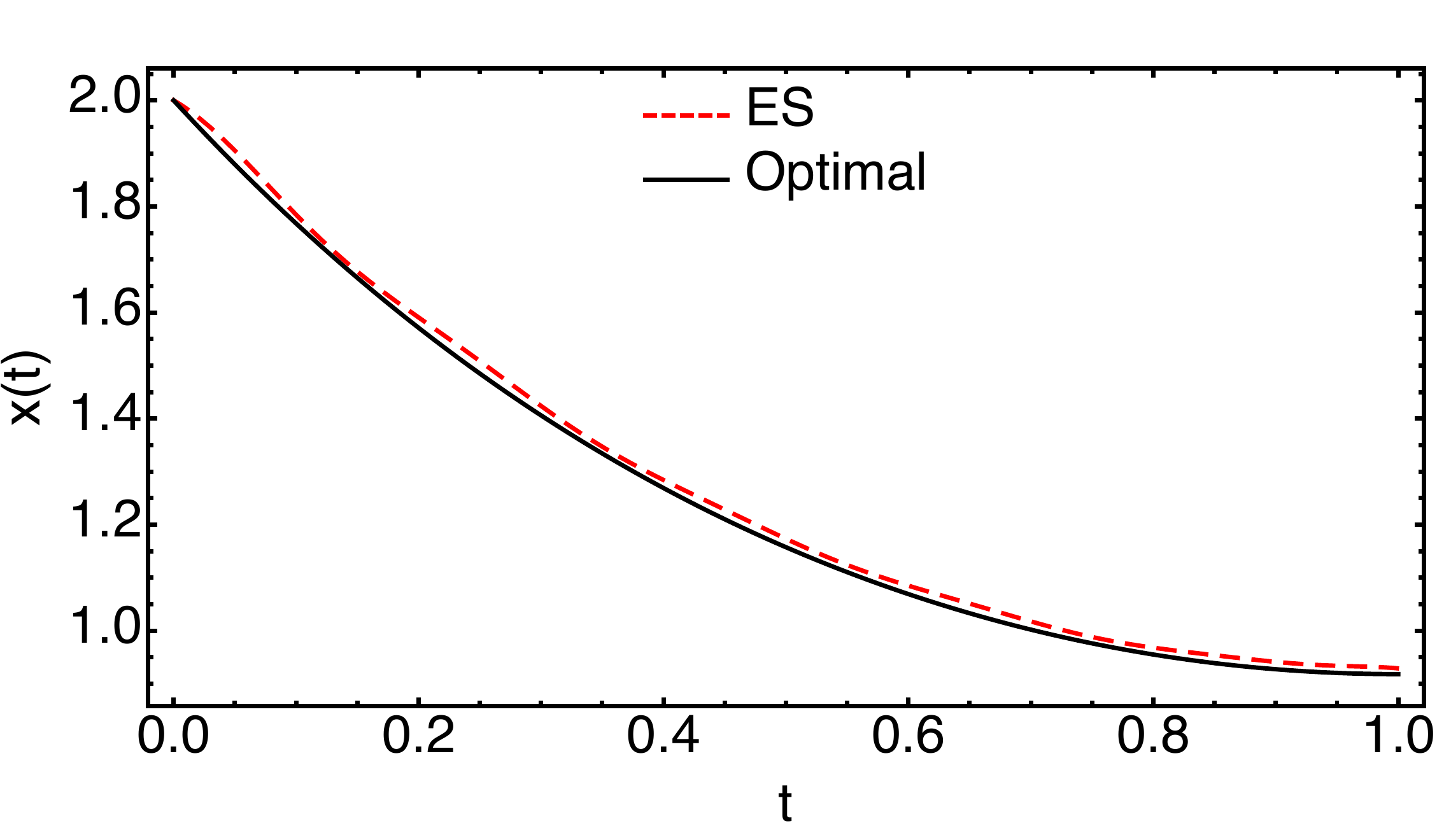} \
    \includegraphics[width=.232\textwidth]{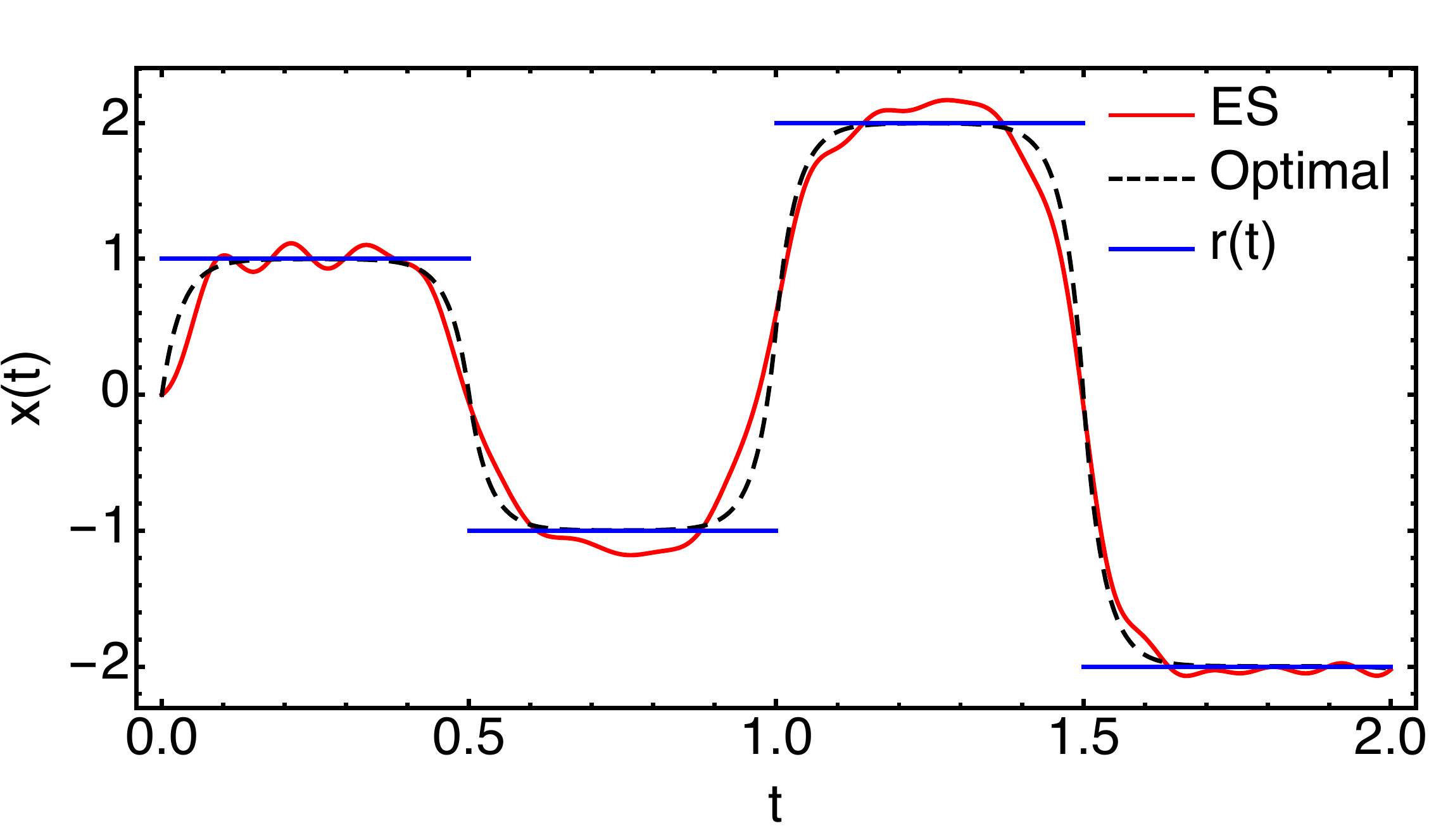} \
    \includegraphics[width=.232\textwidth]{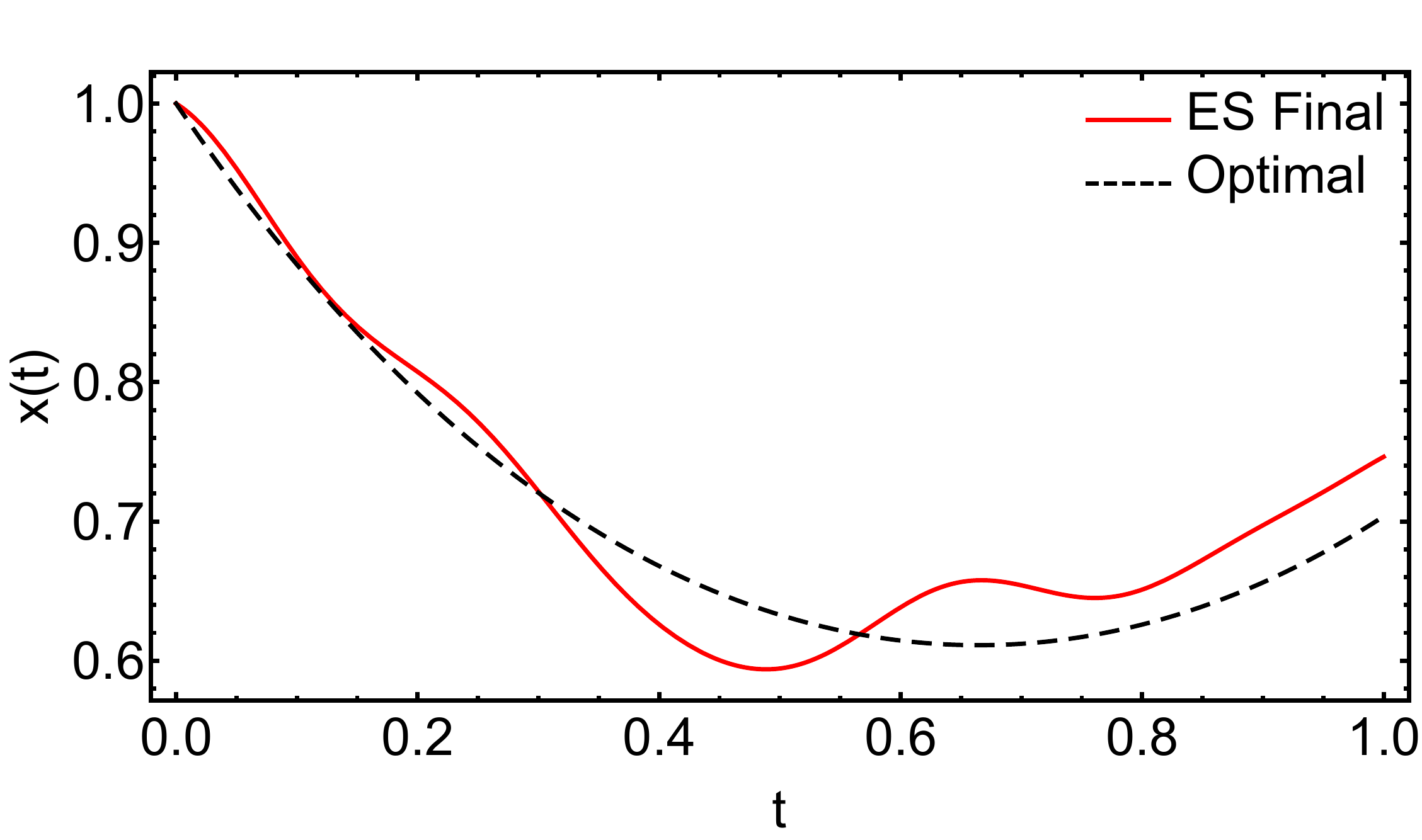} \\
    \includegraphics[width=.232\textwidth]{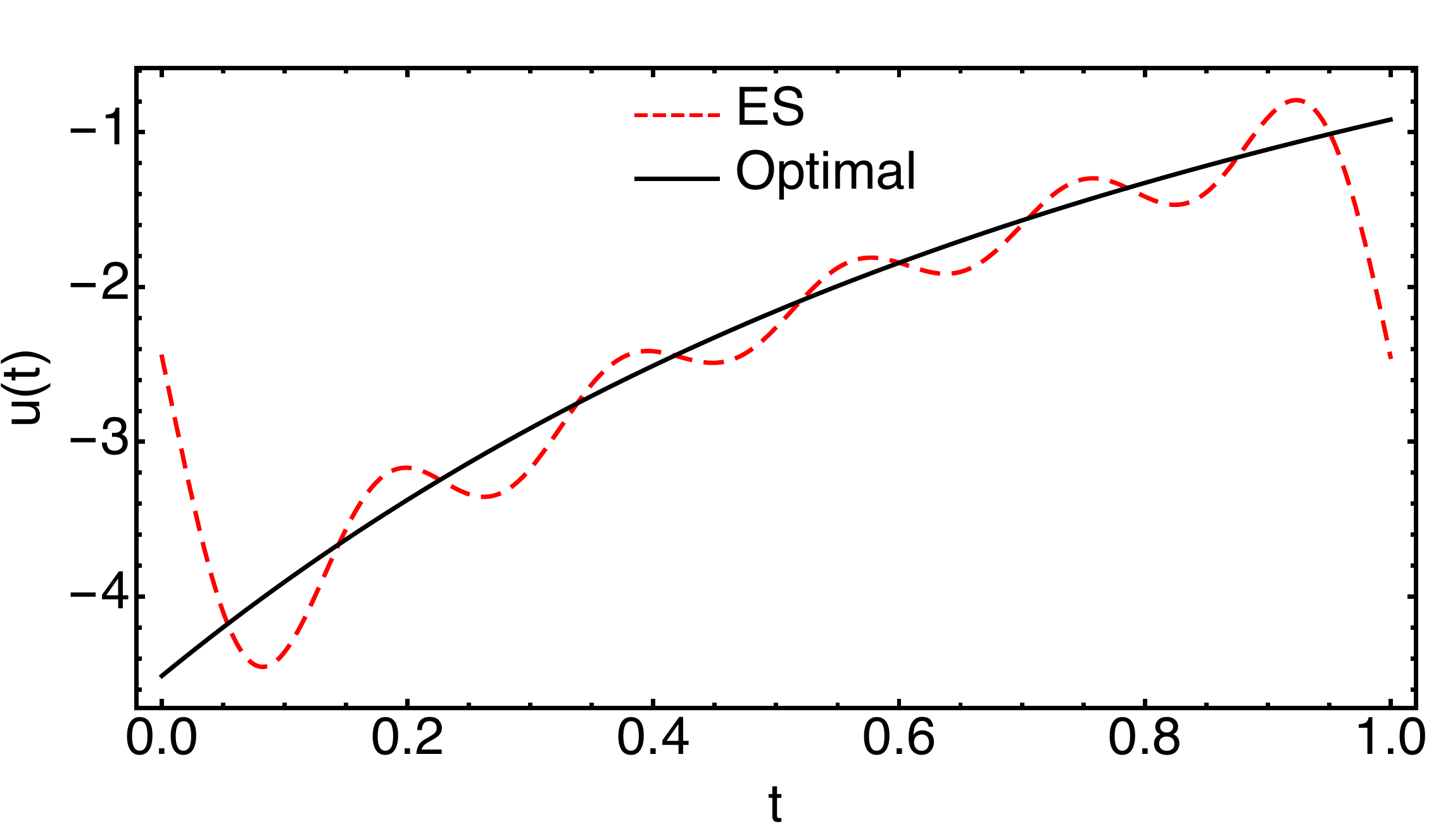} \
    \includegraphics[width=.232\textwidth]{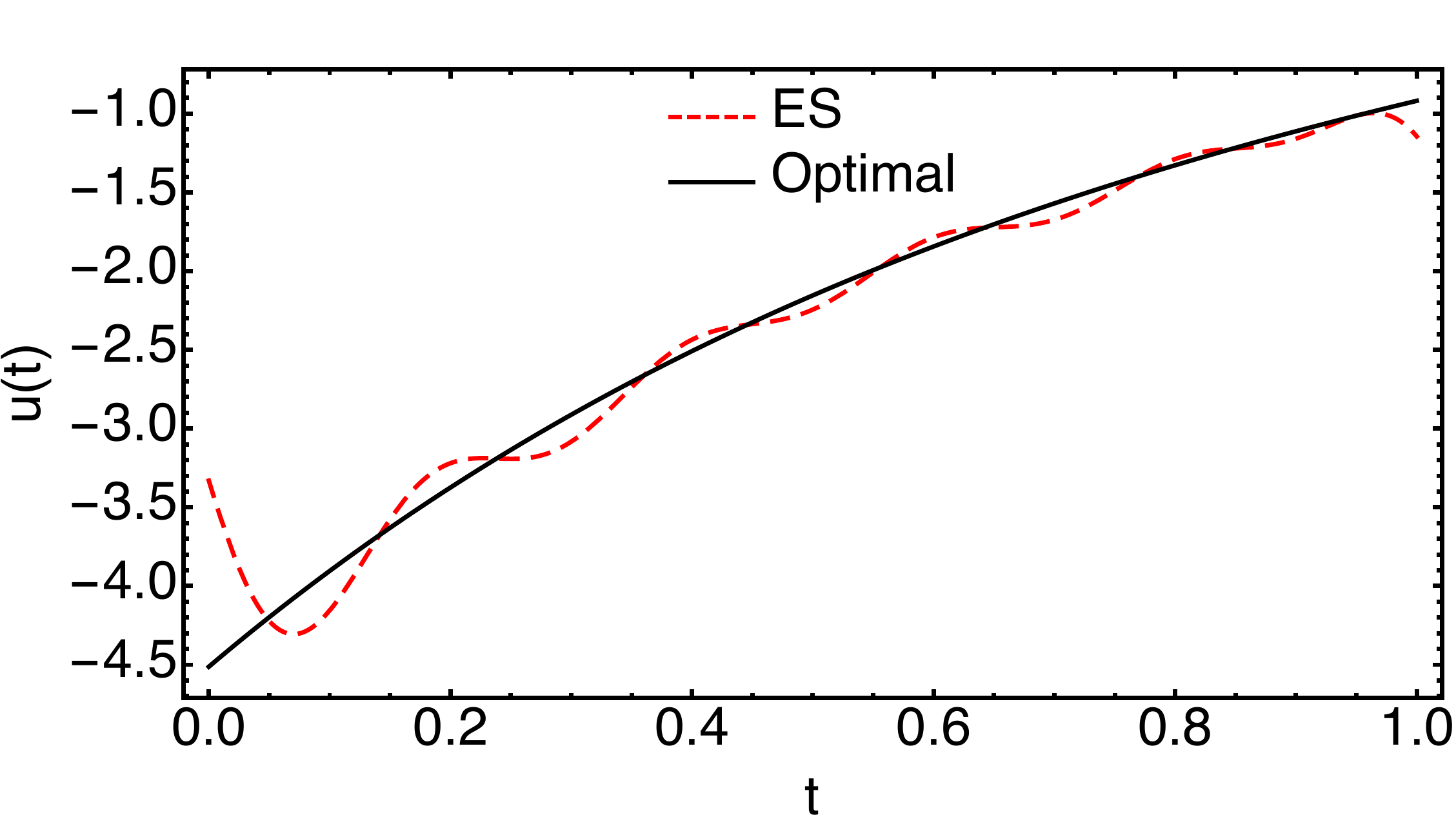} \
    \includegraphics[width=.232\textwidth]{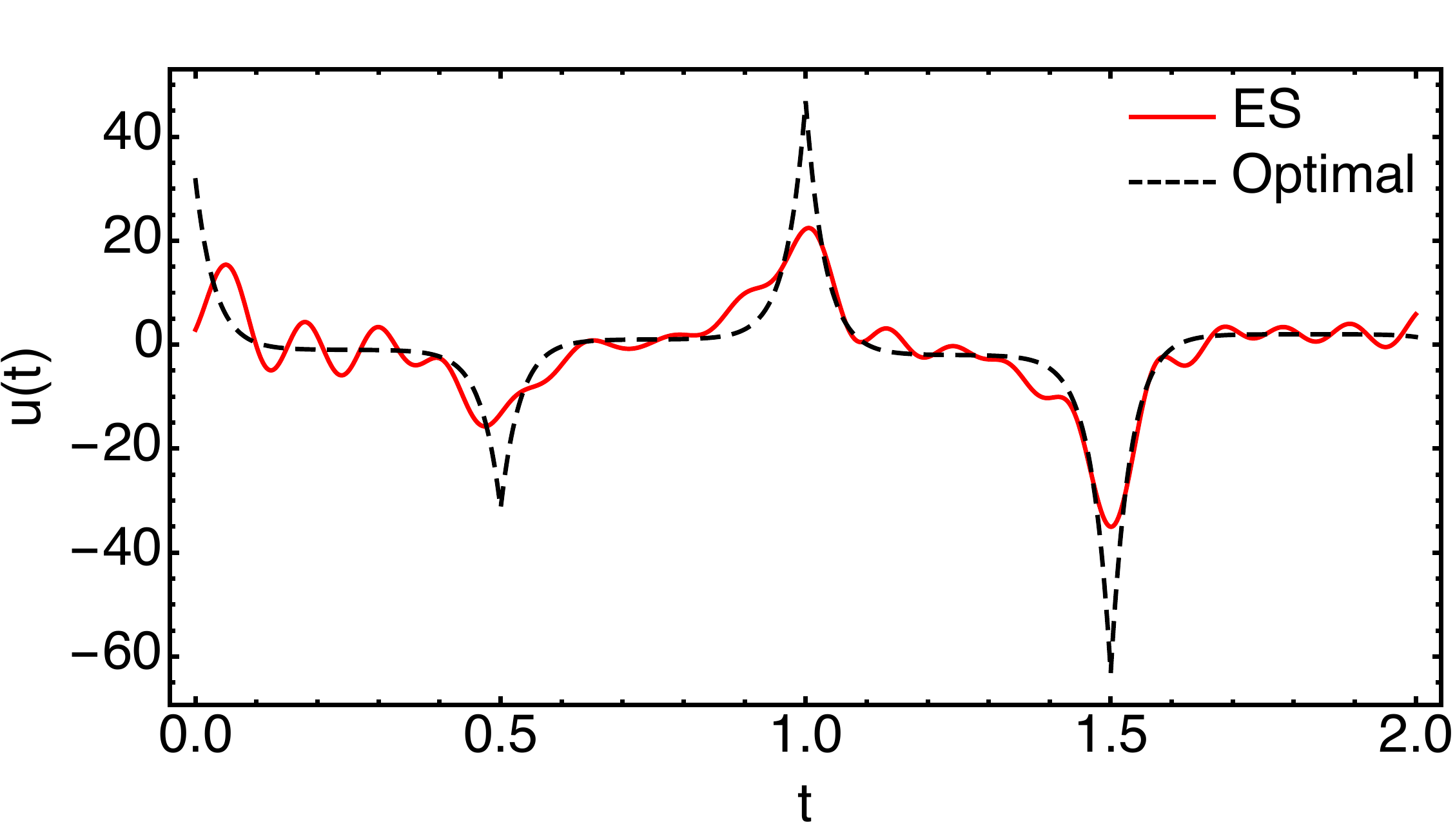} \
    \includegraphics[width=.232\textwidth]{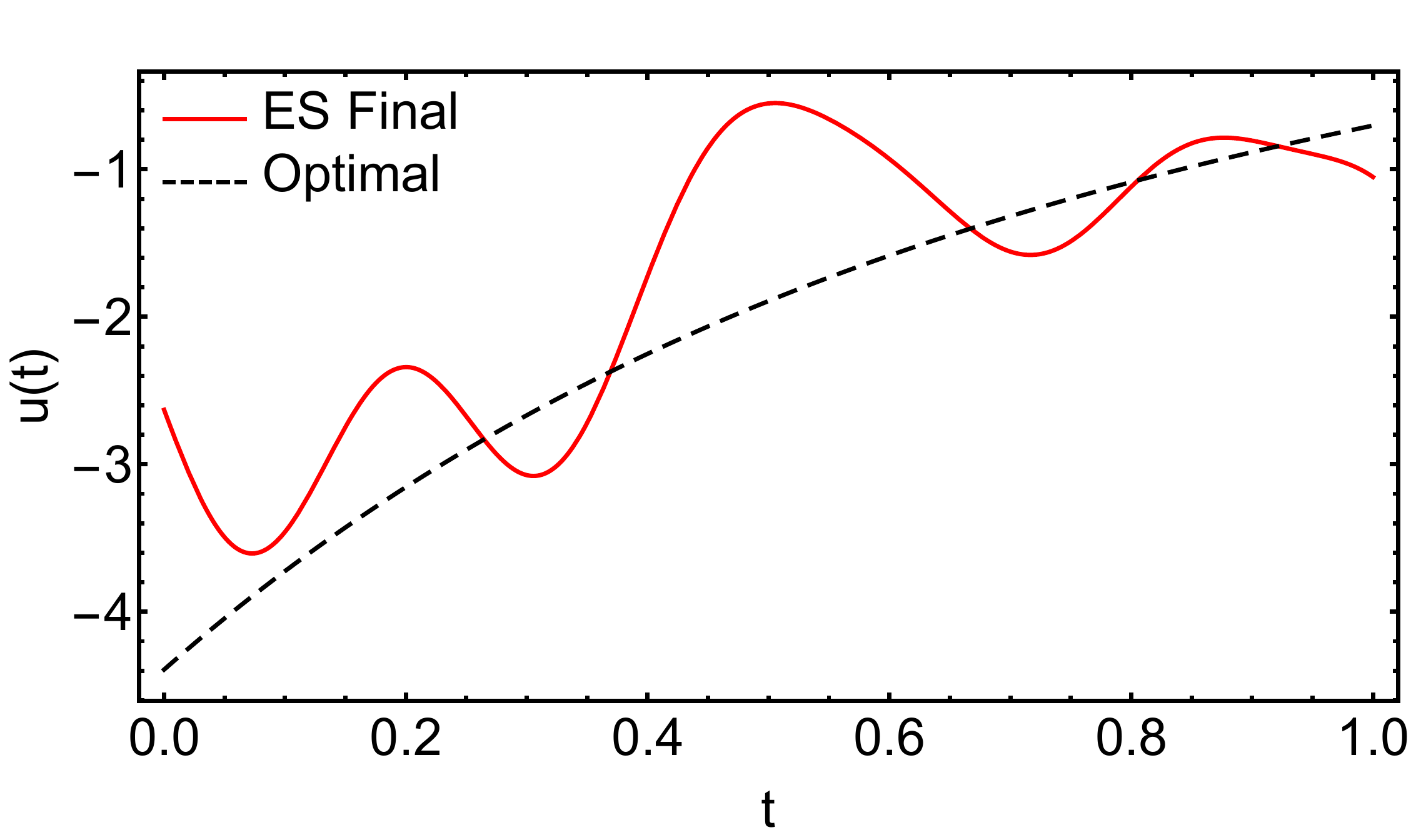} \\
\caption{Results of applying controller (\ref{ES_controller}) for the system in Example \ref{ex_simple} are shown in the first column. The second column shows the same system with controller (\ref{u_example1_v2}). The third and column shows optimal trajectory tracking results for the system in Example (\ref{tvex_simple}) with $m=20$ component controllers. The fourth column shows time varying results for system (\ref{ES_TV}).}
\label{fig_scalar_1}
\end{figure*}

We now focus on linear time-varying systems of the form
\begin{equation}
\frac{d\x(\tau,t)}{d\tau} = A(t)\x(\tau,t) + B(t)\u(\tau,t), \label{LQR_system}
\end{equation}
and their associated cost functions
\begin{equation}
J \ = \ \frac{1}{2} \left ( C\x(T,t) - \r(T) \right )^TP\left( C\x(T,t)-\r(T) \right ) \nonumber
\end{equation}
\begin{equation}
+ \frac{1}{2} \int_{0}^{T} \left (C\x(\tau,t)-\r(\tau) \right )^T Q \left ( C\x(\tau,t) -\r(\tau) \right )d\tau \nonumber
\end{equation}
\begin{equation}
+ \frac{1}{2} \int_{0}^{T} \u(t,\tau)^TR\u(t,\tau) d\tau, \label{LQR_track_cost}
\end{equation}
where $P \geq 0$, $Q \geq 0$, and $R > 0$ are symmetric.

For clarity, we start by stating a simple, scalar result with the particular choice of the Fourier basis over the interval $[0,T]$ with which most readers are familiar. We then present a general basis, full vector-valued result.

%%%
%%%
%%%
\subsection{Time invariant scalar linear quadratic tracker}
%%%
%%%
%%%

Consider a trajectory $\r(t)$, the system (\ref{LQR_system}), and performance index (\ref{LQR_track_cost}). The optimal feedback controller is known to be
\begin{equation}
	\u = -K(t) \x + R^{-1}B^T \v(t), \quad K(t) = R^{-1}B^TS(t), \label{analytic_opt}
\end{equation}
where 
\begin{equation}
	-\dot{S} = A^TS + SA - SBR^{-1}B^TS + C^TQC, S(T)=C^TPC
\end{equation}
and
\begin{equation}
	-\dot{\v} = \left ( A - BK(t) \right )^T\v + C^TQ\r(t), \quad \v(T) = C^TP\r(T). \label{analytic_opt_track}
\end{equation}

We will consider several examples of such systems and demonstrate that the iterative ES scheme converges to the known optimal controller, which we could have designed if we knew the system dynamics exactly.

\begin{example}\label{ex_simple}
For simplicity we start with the system (\ref{LQR_system}) with $A=1$, $B=1$, $x(0)=2$, and $\tau \in [0,1]$, with the objective function defined as $J = x^2(1) + \int_{0}^{1} \left ( x^2(\tau)+u^2(\tau) \right ) d\tau$. We compare the performance of the optimal controller (\ref{analytic_opt}) with the ES algorithm-based controller, a controller which is a linear combination of Fourier basis functions from $L^P[0,T]$:
\begin{equation}
	u(a,b,\tau) = \sum_{j=1}^{m}\left [ a_j(t)\cos\left ( \frac{2\pi j \tau}{T} \right) + b_j(t)\sin\left ( \frac{2\pi j \tau}{T} \right) \right ] \label{ES_controller}
\end{equation}
with $m=5$ Fourier components and evolve the $a_j(t)$ and $b_j(t)$ dynamics according to
\begin{equation}
	\frac{da_j}{dt} = \sqrt{\alpha\omega_j}\cos \left ( \omega_j t + k J \right ), \frac{db_j}{dt} = \sqrt{\alpha\omega_j}\sin \left ( \omega_j t + k J \right )
\end{equation}
with the results shown in the first column of Figure \ref{fig_scalar_1}. A clear limitation of this approach is that by choosing periodic basis functions on the interval $[0,T]$ forces our controller to be periodic, which in this case causes large controller effort overshoot and undershoot at the beginning and end of the time interval. A simple solution to this problem is to use basis functions from a slightly longer time interval, $[0,T+\Delta T]$, so that the controller has more freedom and does not have to be periodic on $[0,T]$, since the last $\Delta T$ time segment does not have any influence on the problem. The controller used in this case is given by:
\begin{equation}
	u = \sum_{j=1}^{5}\left [ a_j(t)\cos\left ( \frac{2\pi j \tau}{T+\Delta T} \right) + b_j(t)\sin\left ( \frac{2\pi j \tau}{T+\Delta T} \right) \right ]. \label{u_example1_v2}
\end{equation}
We choose $\Delta T=0.1$ and re-run the ES optimization, the results are shown in the second column of Figure \ref{fig_scalar_1}.
\end{example}

\begin{example}\label{tvex_simple}
We study a scalar system and find the optimal controller for tracking a time-varying trajectory $r(t)$. The system dynamics are the same as above and the objective function is given by (\ref{LQR_track_cost}), with $C=1$, $P=2$, $Q=20$ and $R=1/50$ which emphasizes tracking and barely penalizes controller effort. Applying controller (\ref{u_example1_v2}) with $m=20$ terms, the results are shown in the third column of Figure \ref{fig_scalar_1}.
\end{example}

\subsection{Time-varying noisy systems}

Next we demonstrate our algorithm's ability to handle time-varying, noisy systems. We simulate the system:
\begin{eqnarray}
	\dot{x} &=& a(t)x + b(t)u, \ \x(n)=1, \ t \in [n,n+1], \label{ES_TV} \\ 
    a(t) &=& 1+t/12000, \\
    b(t) &=& 1 + 0.25 \sin(2\pi t / 3000), \\
    J(n) &=& x^2(n+1) + \int_{n}^{n+1}x^2(t)+u^2(t)dt, \ n \in \mathbb{N}. \label{ES_TV_Cost}
\end{eqnarray}
In this system, $b(t)$ varies sinusoidally with a period of 50 minutes while $a(t)$ increases at a rate of doubling every 3.3 hours. Such time scales are typical for temperature dependent fluctuations of equipment in, for example, a particle accelerator. Our simulation proceeds as: \\
1). We start the simulation at $t=0$ and simulate (\ref{ES_TV})-(\ref{ES_TV_Cost}) over the time interval $t \in [0,1]$, utilizing the controller (\ref{u_example1_v2}), from which we calculate the value $J(0)$ given by (\ref{ES_TV_Cost}). We then record a noise-corrupted measurement, $\hat{J}(0)=J(0)+n(0)$, where $n$ is a mean 0 normally distributed random variable with standard deviation 0.5. \\
2). Utilizing the measurement $\hat{J}(0)$ we update controller parameters to new values, $a_j(1)$ and $b_j(1)$ based on (\ref{a_update}) which defines a new controller. We reset $x$ to $x(1)=1$ and simulate (\ref{ES_TV})-(\ref{ES_TV_Cost}) over the time interval $t \in [1,2]$ to calculate a new objective function value. \\
This process is continued iteratively, as illustrated in Figure \ref{time_scale} with $T=1$. The results are shown in the fourth column of Figure \ref{fig_scalar_1}. The simple time-varying system example described above is illustrative of how our control method could be applied in hardware for the optimal control of the dynamics of a accelerating cavity electromagnetic field, $V(t)$, whose dynamics depend on the geometry of the RF cavity, a time-varying parameter which drifts with temperature $\sim 2\pi \times100$ Hz over the course of a day. The system is initialized from $V(0)=0$, repeatedly for $T=0.001$ seconds at a time, at a rate of 120 times per second, as shown in Figure \ref{time_scale}. Therefore, the system re-starts with a period of $1/120 \sim 0.0083$ seconds, with $0.001$ seconds of operation and $\sim 0.0073$ seconds of off time. For such a system, we would choose initial controller parameters $a_j(0)$, $b_j(0)$, apply control for $t \in [0,0.001]$, calculate the costs, update controller parameters during the approximately $0.0073$ seconds long off time, and then re-initialize the system from 0 and run for another $t \in [0,001]$ seconds.

%%%
%%%
%%%
\subsection{Vector valued linear quadratic tracker}
%%%
%%%
%%%

We now present the general result for vector-valued systems and an arbitrary basis of $L^p[0,T]$. Consider the system
\begin{eqnarray}
	\frac{d\x(\tau,t)}{d\tau} &=& A(t)\x(\tau,t) + B(t)\u(c(t),\tau), \label{GenLQR} \\
	u_i(\c(t),\tau) &=& \sum_{j=1}^{m}c_{i,j}(t)\varphi_j(\tau), \label{GenLQRu}
\end{eqnarray}
where $\x(0,t) = \x_0$, the functions $\varphi_j(\tau): \mathbb{R} \rightarrow \mathbb{R}$ are any subset of any choice of basis $\beta$ of $L^p[0,T]$, and coefficients 
\begin{equation}
	\c = \left ( \c_1, \dots, \c_n \right ) \in \mathbb{R}^{n\times m}, \quad \c_i = \left ( c_{i,1},\dots,c_{i,m} \right ) \in \mathbb{R}^m.
\end{equation}
Consider a trajectory $\r(t) : \mathbb{R} \rightarrow \mathbb{R}^n$, 
and the cost function (\ref{LQR_track_cost}). The coefficients $\c_i$ have dynamics
\begin{equation}
    \frac{dc_{i,j}(t)}{dt} = \sqrt{\alpha\omega_{i,j}}\cos \left ( \omega_{i,j} t + k J(\c(t),t) \right ). \label{ccdotn}
\end{equation} 
For large $\omega_0$, the average coefficient dynamics are
\begin{equation}
	\dot{\bar{c}}_{i,j} = -\frac{k\alpha}{2}\frac{\partial J(\bar{\c},t)}{\partial \bar{c}_{i,j}}. \label{cdotLQRT}
\end{equation}
Because $J$ is convex relative to any basis, $\beta$ of $L^p[0,T]$, the $\c(t)$ converge towards values which give us the optimal controller minimizing $J$ over the subspace of $L^p[0,T]$ spanned by the basis vectors $\left \{ \varphi_1(t), \dots, \varphi_m(t) \right \}$. In particular, consider the following Corollary to Theorem \ref{Optimality}.

\begin{corollary}(Of Theorem \ref{Optimality})
For the time-dependent Linear Quadratic Tracker optimal control problem (\ref{LQR_track_cost}, (\ref{GenLQR}), the cost $\{J(a_\omega(t)\}$ of the controllers $\{u(a_\omega(t)\}$ defined by system (\ref{cdotLQRT}) converges to $min(J)$.
\end{corollary}

\begin{proof}
Consider the time-dependent Linear Quadratic Tracker. To see that the cost function $J$ is convex in $u$ recall that (1) affine functions are convex, (2) a quadratic form $x^tQx$ is convex whenever $Q$ is positive semidefinite, (3) if $f$ is convex and $g$ is affine, then their composition is $f\circ g$ is convex, and (4) that $(C \x(u)(T) - \r(T) )^TP (C\x(u)(T)-\r(T)$ is the composition of the convex function $\y \to \y^TP\y$ with the affine function $\y \rightarrow (C\y-r)$ with the linear function $u\rightarrow \x(u)$. Similar reasoning establishes the convexity of the other components of $J$. The result follows from Theorem \ref{Optimality}.
\end{proof}

\subsection{Feedback control when state measurement $x(t)$ available}

\begin{figure}[!t]
	\centering
    \includegraphics[width=.35\textwidth]{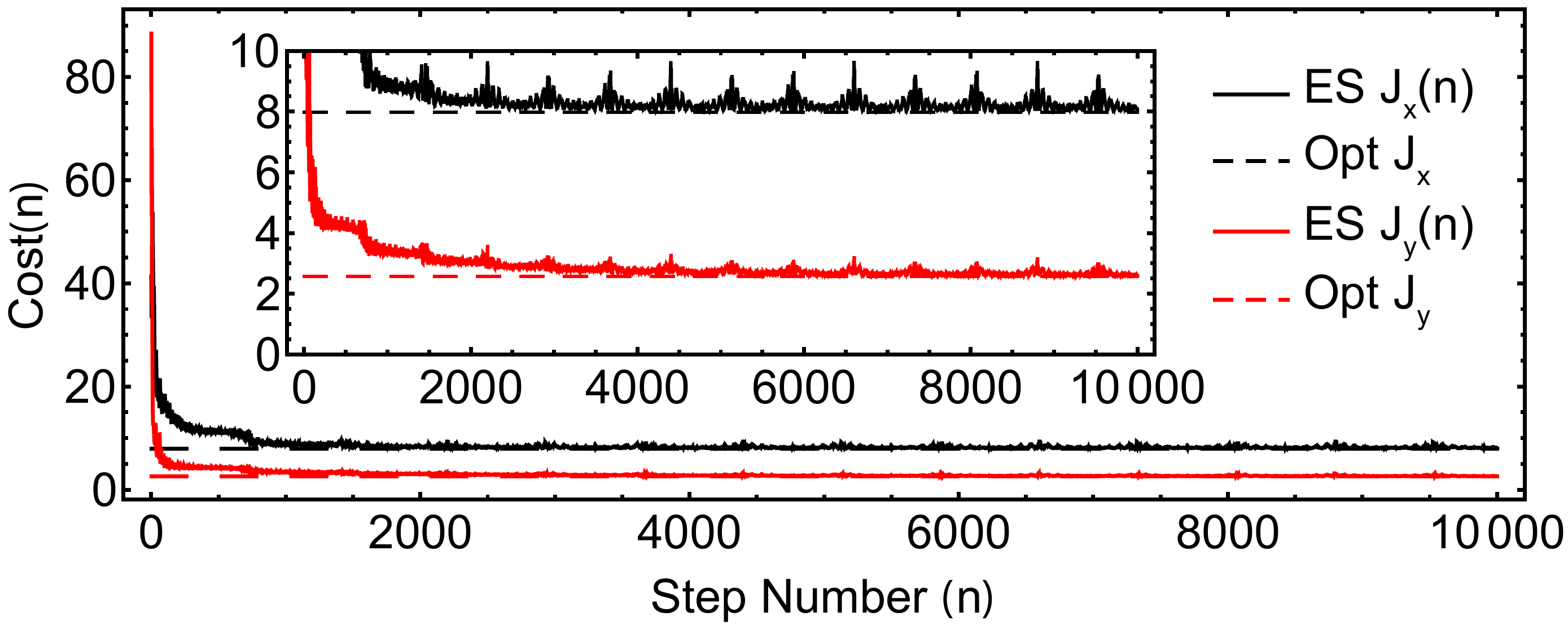}
    \includegraphics[width=.35\textwidth]{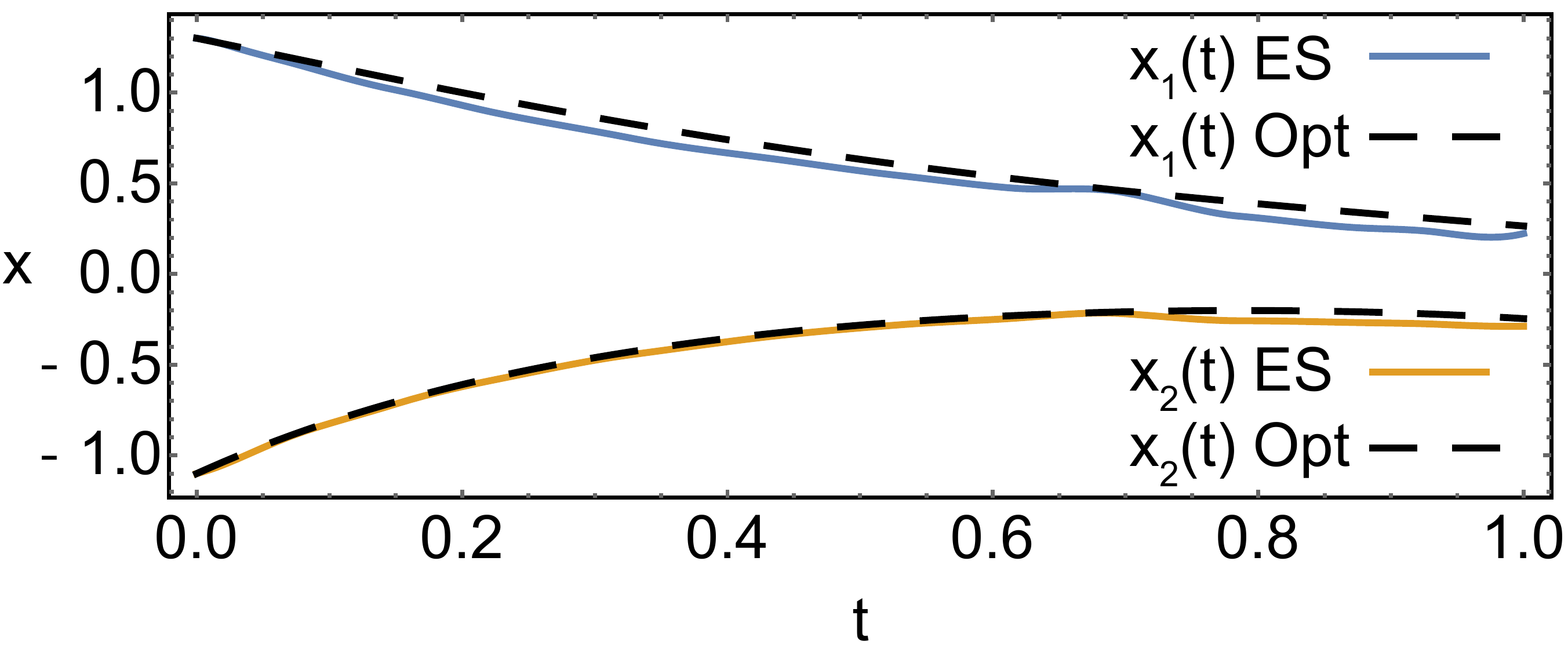}
    \includegraphics[width=.35\textwidth]{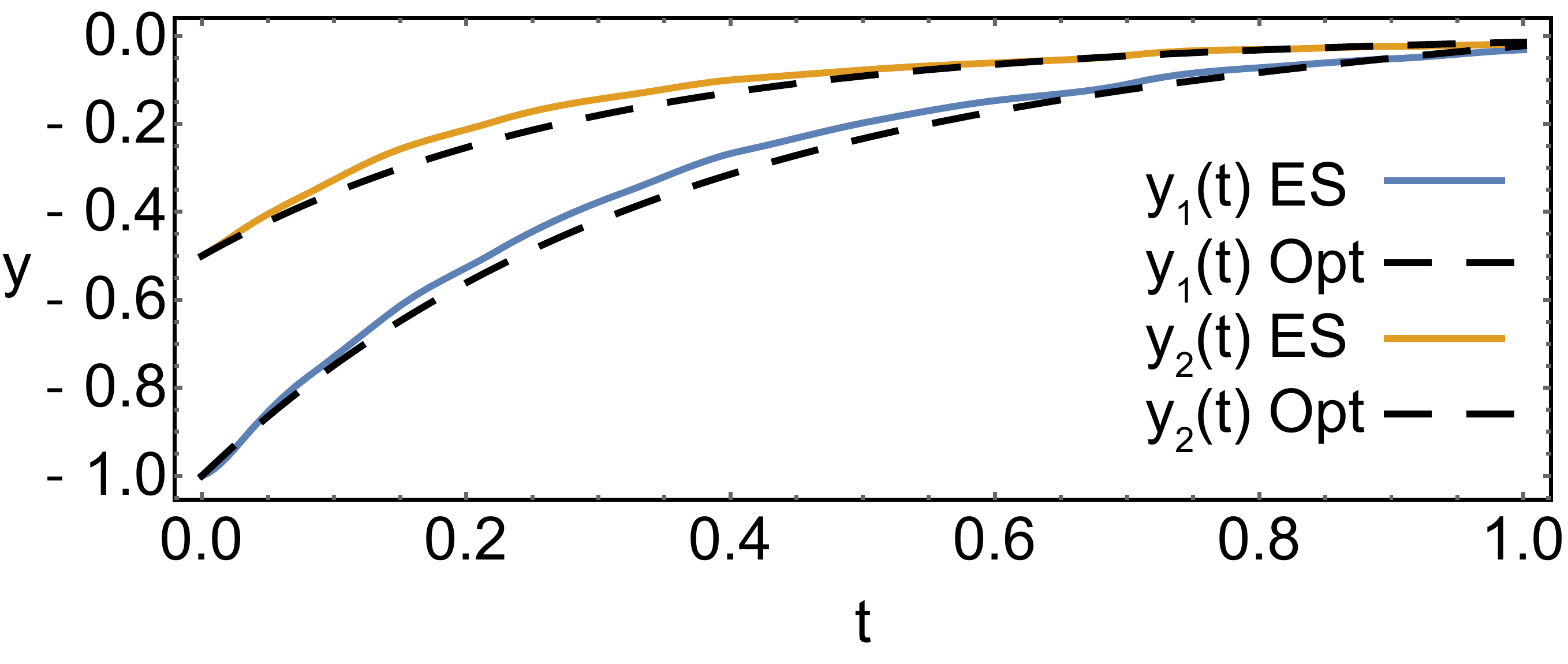}
    \includegraphics[width=.35\textwidth]{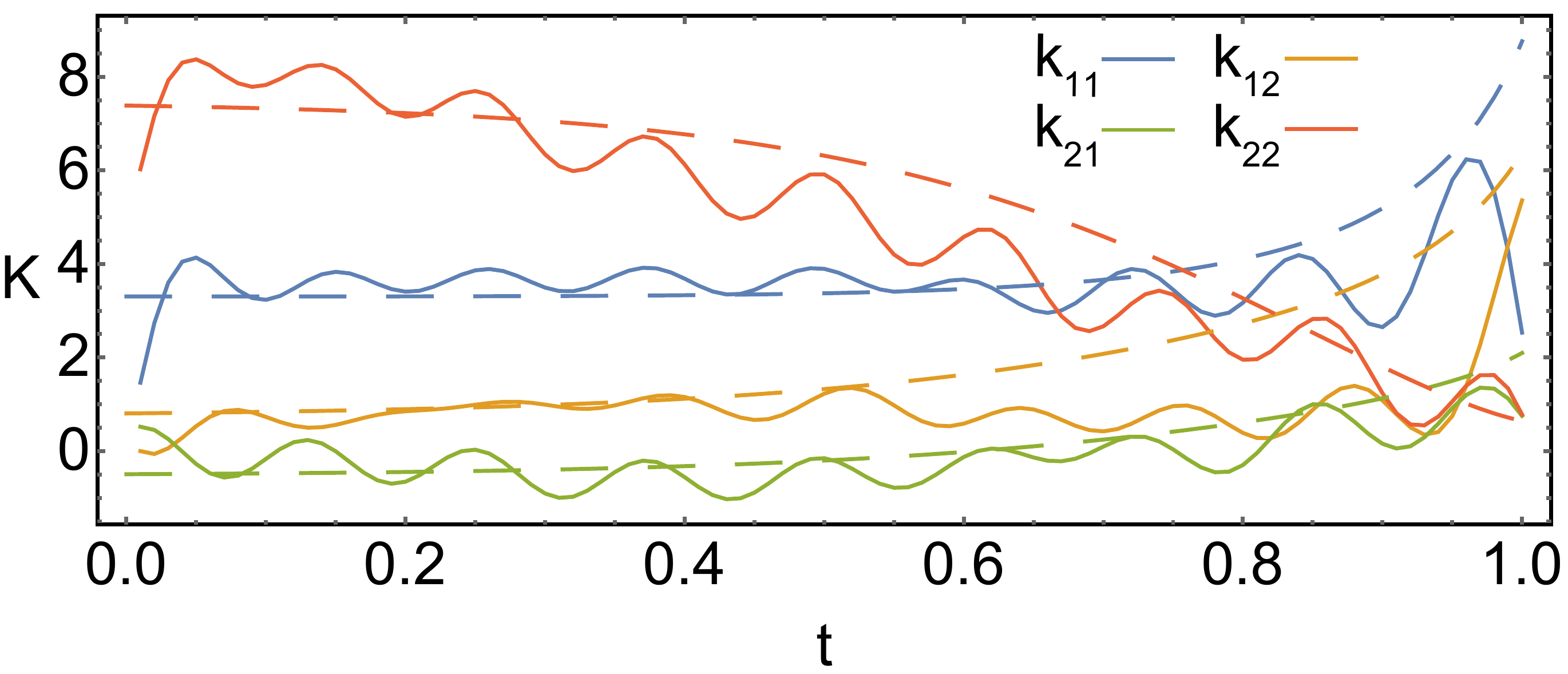}
    \includegraphics[width=.35\textwidth]{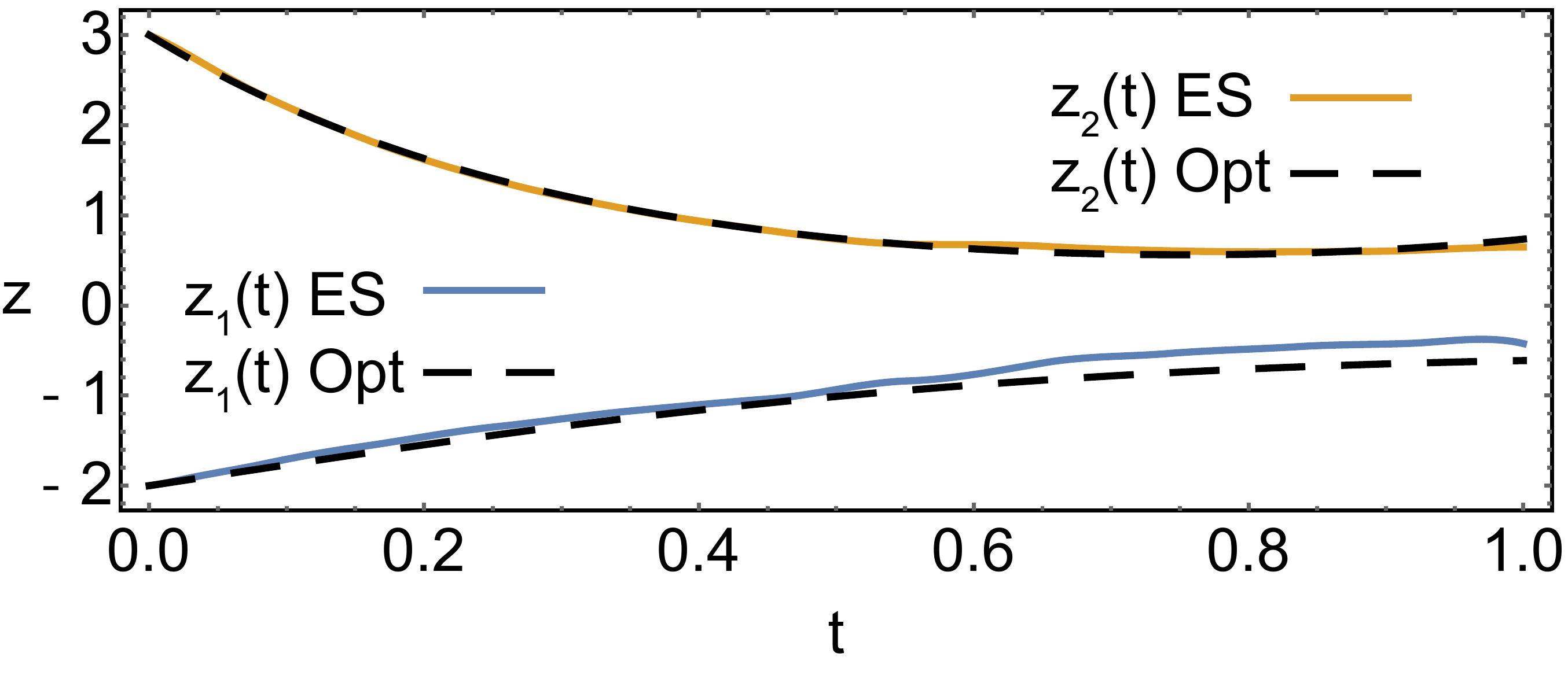}
\caption{The top plot shows the evolution of the objective functions $J_x$ and $J_y$ for the same system with two different initial conditions. The second and third plots down show the resulting trajectories alongside the analtically known optimal trajectories. The fourth plot shows the $K_{\mathrm{ES}}(t,\tau)$ to which the ES scheme has converged alongside the analytically determined $K(t)$ if the system and objective funtion had been known. Finally, the last plot shows the results of applying the feedback $-K_{\mathrm{ES}}(t,\tau)x(t)$ for a third initial condition, for system (\ref{z_sys}).}
\label{fig_2D}
\end{figure}

For a linear system, without tracking, we know that the optimal controller is of the form $\u(\tau) = -K(\tau) x(\tau)$ for $\tau \in [t_0,t_0+T]$, with $K$ given according to (\ref{analytic_opt}). If we have access to full state measurements $x(\tau)$, we can then design our controller in such a way that we find the $K(\tau)$ matrix directly. For an $n$-dimensional system, there are $n^2$ unknowns in the $K(\tau)$ matrix at any time $\tau$, therefore, we must solve the optimal control problem, (\ref{LQR_system}), (\ref{LQR_track_cost}) for $\geq n$ different initial conditions. We demonstrate this with a 2-dimensional example. We repeatedly simulate the following two systems for $\tau \in [0,1]$
\begin{equation}
	\dot{x} = Ax + Bu_x, \quad u_x = -K_{\mathrm{ES}}(t,\tau)x, \ x(0) = (x_{1,0},x_{2,0}), \label{2D_sys_ex}
\end{equation}
\begin{equation}
	\dot{y} = Ay + Bu_y, \quad u_y = -K_{\mathrm{ES}}(t,\tau)y, \ y(0) = (y_{1,0},y_{2,0}),
\end{equation}
with feedback gain matrix $K_{\mathrm{ES}} = \left \{ k_{l,p}(t,\tau) \right \}$,
\begin{equation}
	k_{l,p}(t,\tau) = \sum_{j=1}^{m}\left [ a^{l,p}_j(t)\cos\left ( \nu_j \tau\right) + b^{l,p}_j(t)\sin\left ( \nu_j \tau \right) \right ],  \label{K_controller}
\end{equation}
where as before $\nu_j = \frac{2\pi j}{T+\Delta T}$ and we update the $a^{l,p}_j(t)$, $b^{l,p}_j(t)$ according to (\ref{ccdotn}), where the objective function being minimized is now the sum $J=J_x+J_y$, where
\begin{eqnarray}
J_x &=& \frac{x^TPx}{2} + \frac{1}{2} \int_{0}^{T} \left ( x^TQx + u^T_x Ru_x \right ) d\tau, \label{LQR_2D_examplex} \\
J_y &=& \frac{y^TPy}{2} + \frac{1}{2} \int_{0}^{T} \left ( y^TQy + u^T_yRu_y \right ) d\tau. \label{LQR_2D_exampley}
\end{eqnarray}
Once the algorithm has converged and we have found a good $K_{\mathrm{ES}}(t,\tau)$, we now have an optimal feedback controller for the unknown system (\ref{2D_sys_ex}) relative to the unknown objective function (\ref{LQR_2D_examplex}), for all initial conditions. For example, consider
\begin{equation}
	A = \left[\begin{array}{cc}
    1.0 & 0.25 \\
    0.3 & 0.7
    \end{array}\right], 
    B = \left[\begin{array}{cc}
    1.0 & 0.1 \\
    0.2 & 0.5
    \end{array}\right], 
    P = \left[\begin{array}{cc}
    4.0 & 3.0 \\
    3.0 & 1.0
    \end{array}\right], \nonumber
\end{equation}
\begin{equation}
    Q = \left[\begin{array}{cc}
    2.0 & 0.1 \\
    0.1 & 10
    \end{array}\right],
    R = \left[\begin{array}{cc}
    0.5 & 0.1 \\
    0.1 & 0.25
    \end{array}\right],
\end{equation}
$x(0) = (1.3,-1.1)$, and $y(0)=(-1,-0.5)$. The $K_{\mathrm{ES}}$ matrix is found iteratively with $m=10$, by
\begin{eqnarray}
    a^{1,1}_j(s+1) &=& a^{1,1}_j(s) + \Delta \sqrt{\alpha \omega^{1}_{a,j}}\cos \left ( \omega^{1}_{a,j} s \Delta + k J(s) \right ), \nonumber \\
    b^{1,1}_j(s+1) &=& b^{1,1}_j(s) + \Delta \sqrt{\alpha \omega^{1}_{b,j}}\cos \left ( \omega^{1}_{b,j} s \Delta + k J(s) \right ), \nonumber \\
    a^{1,2}_j(s+1) &=& a^{1,2}_j(s) + \Delta \sqrt{\alpha \omega^{1}_{a,j}}\sin \left ( \omega^{1}_{a,j} s \Delta + k J(s) \right ), \nonumber \\
    b^{1,2}_j(s+1) &=& b^{1,2}_j(s) + \Delta \sqrt{\alpha \omega^{1}_{b,j}}\sin \left ( \omega^{1}_{b,j} s \Delta + k J(s) \right ), \nonumber
\end{eqnarray}
with the frequencies $\left \{ \omega^{1}_{a,1}, \dots, \omega^{1}_{a,m}, \omega^{1}_{b,1},\dots,\omega^{1}_{b,m} \right \}$ evenly distributed between $\omega_0$ and $1.35 \omega_0$. Similary, the $a^{2,1}_j$, $b^{2,1}_j$ are updated with $\cos()$ while $a^{2,2}_j$, $b^{2,2}_j$ are updated with $\sin()$, using $\omega^{2}_{a,j}$, $\omega^{2}_{b,j}$, where the frequencies $\left \{ \omega^{2}_{a,1}, \dots, \omega^{2}_{a,m}, \omega^{2}_{b,1},\dots,\omega^{2}_{b,m} \right \}$ are evenly distributed between $1.35\omega_0$ and $1.75 \omega_0$. We used $\omega_0 = 3197$, $\Delta = 2\pi /( 10 \times 1.75 \omega_0)$, $\alpha=320$, and $k=0.1$. Once the ES parameters have converged, we compare $K_{\mathrm{ES}}$ to the analytically known optimal $K$ and we apply the ES-based feedback control to a third system with new initial conditions:
\begin{equation}
	\dot{z} = Az + Bu_z, \quad u_z = -K_{\mathrm{ES}}(t,\tau)z, \quad z(0) = (z_{1,0},z_{2,0}), \label{z_sys}
\end{equation}
with $z(0) = (-2,3)$, and compare the trajectories to those obtained by the optimal controller $u = -Kz$. The simulation results are summarized in Figure (\ref{fig_2D}).

For a fixed desired trajectory, $r(t)$, according to equations (\ref{analytic_opt}), (\ref{analytic_opt_track}), the optimal feedback control is given by $u=-K(\tau)x(\tau) + R^{-1}B^Tv(\tau)$, where $v(\tau)$ depends on system and objective function parameters and $r(\tau)$, but not on the initial condition value $x(0)$. Therefore, in a similar fashion as done above, we can iteratively create an optimal tracking feedback controller independent of the initial conditions by creating a controller of the form:
\begin{equation}
	u = -K_{\mathrm{ES}}(\tau)x(\tau) + V_{\mathrm{ES}}(\tau).
\end{equation}
This can be applied to linear systems of any dimensions.

%
%%
%%%
%%%%
%%%%%
%%%%%%
%%%%%%%
%%%%%%%%%%%%%%%%%%%%%%%%%%%%%%%%%%%%%%%%%%%%%%%%%%%%%%%%%%%%%%%%%%%%%%%%%%%%%%%%%%%%%%%%
%%%%%%%%%%%%%%%%%%%%%%%%%%%%%%%%%%%%%%%%%%%%%%%%%%%%%%%%%%%%%%%%%%%%%%%%%%%%%%%%%%%%%%%%

\section{Conclusions}\label{sec:con}

%%%%%%%%%%%%%%%%%%%%%%%%%%%%%%%%%%%%%%%%%%%%%%%%%%%%%%%%%%%%%%%%%%%%%%%%%%%%%%%%%%%%%%%%
%%%%%%%%%%%%%%%%%%%%%%%%%%%%%%%%%%%%%%%%%%%%%%%%%%%%%%%%%%%%%%%%%%%%%%%%%%%%%%%%%%%%%%%%
%%%%%%%
%%%%%%
%%%%%
%%%%
%%%
%%
%

We present an iterative ES algorithm for creating optimal controllers for unknown, time-varying systems based only on noisy measurements of analytically unknown objective functions. This algorithm can be useful for application in hardware for systems that are repeatedly initialized from the same initial conditions and must perform a given task despite uncertain time variation of the system components, such as drifts due to aging and thermal cycling.

%
%%
%%%
%%%%
%%%%%
%%%%%%
%%%%%%%
%%%%%%%%%%%%%%%%%%%%%%%%%%%%%%%%%%%%%%%%%%%%%%%%%%%%%%%%%%%%%%%%%%%%%%%%%%%%%%%%%%%%%%%%
%%%%%%%%%%%%%%%%%%%%%%%%%%%%%%%%%%%%%%%%%%%%%%%%%%%%%%%%%%%%%%%%%%%%%%%%%%%%%%%%%%%%%%%%

\end{document}